\documentclass[10pt]{amsart}

\usepackage[margin=1in]{geometry}
\usepackage{amsmath}
\usepackage{booktabs}
\usepackage{tabularx}

\usepackage{amsmath,amssymb,amsfonts,amsthm}
\usepackage{enumerate,mathrsfs,latexsym, comment, mathtools, mathdots} 
\usepackage{multirow}
\usepackage{caption}
\usepackage{subcaption}
\usepackage{hyperref}
\usepackage{enumitem}

\usepackage{float}

\usepackage{tikz-cd}
\usepackage{tikz}
\usetikzlibrary{snakes, 
	3d, matrix,decorations.pathreplacing,calc,decorations.pathmorphing, patterns, shapes.misc}
\usetikzlibrary {positioning}

\numberwithin{equation}{section}
\allowdisplaybreaks

\makeatletter
\newcommand{\leqnomode}{\tagsleft@true\let\veqno\@@leqno}
\newcommand{\reqnomode}{\tagsleft@false\let\veqno\@@eqno}
\makeatother

\newcommand{\C}{{\mathbb{C}}}

\DeclareMathOperator{\SL}{SL}

\DeclareMathOperator{\Gr}{Gr}
\DeclareMathOperator{\Hol}{Hol}
\DeclareMathOperator{\SU}{SU}
\DeclareMathOperator{\Sp}{Sp}
\DeclareMathOperator{\Spin}{Spin}
\DeclareMathOperator{\SO}{SO}

\DeclareMathOperator{\gr}{gr}

\DeclareMathOperator{\dex}{dex}
\DeclareMathOperator{\rank}{rank}
\DeclareMathOperator{\IGr}{IGr}
\DeclareMathOperator{\OGr}{OGr}

\newcommand{\Exterior}{\mathchoice{{\textstyle\bigwedge}}%
	{{\bigwedge}}%
	{{\textstyle\wedge}}%
	{{\scriptstyle\wedge}}}

\tikzstyle{point} = [
draw=black,
cross out,
inner sep=0pt,
minimum width=4pt,
minimum height=4pt]

\tikzstyle{double line} = [
double distance = 1.5pt, 
double=\pgfkeysvalueof{/tikz/commutative diagrams/background color}
]

\tikzstyle{state}=[draw, circle, inner sep = 0.06cm]

\tikzstyle{triple line} = [
	double distance = 2pt, 
	double=\pgfkeysvalueof{/tikz/commutative diagrams/background color}
]

\newtheorem{theorem}{Theorem}[section]
\newtheorem{lemma}[theorem]{Lemma}
\newtheorem{proposition}[theorem]{Proposition}
\newtheorem{corollary}[theorem]{Corollary}

\theoremstyle{definition}
\newtheorem{example}[theorem]{Example}
\newtheorem{definition}[theorem]{Definition}

\theoremstyle{remark}
\newtheorem{remark}[theorem]{Remark}

\begin{document}
	
\title[Complete intersection hyperk\"{a}hler fourfolds in rational homogeneous varieties]{Complete intersection hyperk\"{a}hler fourfolds \\ with respect to equivariant vector bundles \\ over rational homogeneous varieties of Picard number one}

\author{Eunjeong Lee}
\address{Department of Mathematics, Chungbuk National University, Cheongju 28644, Republic of Korea}
\email{eunjeong.lee@chungbuk.ac.kr}

\author{Kyeong-Dong Park}
\address{Department of Mathematics and Research Institute of Molecular Alchemy, Gyeongsang National University, 501 Jinju-daero, Jinju 52725, Republic of Korea}
\email{kdpark@gnu.ac.kr} 


\begin{abstract}
We classify fourfolds with trivial canonical bundle which are zero loci of general global sections of completely reducible equivariant vector bundles over exceptional homogeneous varieties of Picard number one. 
By computing their Hodge numbers, 
we see that there exist no hyperk\"{a}hler fourfolds among them. 
This implies that a hyperk\"{a}hler fourfold represented as the zero locus of a general global section of a completely reducible equivariant vector bundle over a rational homogeneous variety of Picard number one is one of the two cases described by Beauville--Donagi and Debarre--Voisin. 
\end{abstract}
\maketitle
\date{\today}

\section{Introduction}
The holonomy group of the Levi-Civita connection on a Riemannian manifold is a Lie group given by parallel transports along loops based at a fixed point, and is a global invariant which measures constant tensors on the manifold. 
Riemannian manifolds with special holonomy groups are interesting in various aspects, and they include Calabi--Yau, hyperk\"{a}hler, and quaternionic K\"{a}hler manifolds. 
Recall that a $d$-dimensional Riemannian manifold $(M, g)$ is called \emph{Calabi--Yau}, \emph{hyperk\"{a}hler}, or \emph{quaternionic K\"{a}hler} 
if its holonomy group $\Hol(M, g)$ is $\SU(n), \Sp(m)$, or $\Sp(m) \cdot \Sp(1)$, respectively, where $d=2n$, or $4m$. 
However, it is extremely hard to construct Riemannian manifolds with a specific holonomy; see, for example,~\cite{Besse87} for details. 
For instance, there are no known examples of compact quaternionic K\"{a}hler manifolds that are not symmetric spaces.
It is also a difficult problem to provide explicit examples of compact hyperk\"{a}hler manifolds. 

Known examples of compact hyperk\"{a}hler manifolds are K3 surfaces, the Hilbert schemes of $n$ points on K3 surfaces, and generalized Kummer varieties of abelian surfaces described by Beauville~\cite{Beauville83}. 
We also have two explicit descriptions for compact hyperk\"{a}hler manifolds of complex dimension 4 due to Beauville--Donagi~\cite{BD} and Debarre--Voisin~\cite{DV} 
even though these are deformation-equivalent to the Hilbert square of a K3 surface of genus 8 and 12, respectively. 
Beauville and Donagi proved that the variety of complex lines on a smooth cubic hypersurface in the complex projective space $\mathbb P^5$ is hyperk\"{a}hler, 
and it can be described as the zero locus of a general global section of the third symmetric power $S^3 \mathcal U^*$ of the dual universal subbundle of the complex Grassmannian $\Gr(2, 6)$. 
Debarre and Voisin gave another example of compact hyperk\"{a}hler fourfolds as the zero locus of a general global section of the third exterior power $\Exterior^3 \mathcal U^*$ of the dual universal subbundle of the Grassmannian $\Gr(6, 10)$. 
We might expect to find another example of compact hyperk\"{a}hler fourfolds arising as the zero loci of general global sections of equivariant vector bundles over rational homogeneous varieties $G/P$. 
Here, $G$ is a complex semisimple Lie group and $P \subset G$ is a maximal parabolic subgroup, that is, $P$ is a maximal proper parabolic subgroup with respect to inclusion.  
We note that $G/P$ is a rational homogeneous variety of Picard number one if and only if $P$ is maximal. 
However, the actual results are far from this expectation.

Recently, Benedetti~\cite{Ben} proved that 
if $Z$ is a hyperk\"{a}hler fourfold 
which is the zero locus of a general global section of completely reducible, globally generated, equivariant vector bundles over a Grassmannian, or an isotropic (that is, orthogonal or symplectic) Grassmannian, 
then $Z$ is  either of Beauville--Donagi type or of Debarre--Voisin type. 
This provides a classification of hyperk\"{a}hler fourfolds which are zero loci of general global sections of completely reducible equivariant vector bundles over rational homogeneous varieties of the \emph{classical} Lie groups, that is, of type $A, B, C, D$.
Being motivated by the above works, 
we classify hyperk\"{a}hler fourfolds appearing in rational homogeneous varieties of the \emph{exceptional} Lie groups, that is, of type $E, F, G$. 
See Table~\ref{Dynkin} for the Dynkin diagrams of type $E, F, G$.
More precisely, we obtain the following result. 

\begin{theorem} 
\label{hyperkahler fourfold}
Let $G/P$ be a rational homogeneous variety of Picard number one, where $G$ is a complex simple Lie group of exceptional type and $P \subset G$ is a maximal parabolic subgroup.
For a completely reducible, globally generated, equivariant vector bundle $\mathcal F$ over $G/P$, 
if $Z \subset G/P$ is a fourfold which is the zero locus of a general global section of $\mathcal F$, 
then $Z$ is not hyperk\"{a}hler.
\end{theorem}

Together with Benedetti's result~\cite{Ben}, this theorem immediately implies the following. 

\begin{corollary}
\label{corollary of main result}
Let $\mathcal F$ be a completely reducible, globally generated, equivariant vector bundle over a rational homogeneous variety $G/P$ of Picard number one.
If $Z \subset G/P$ is a hyperk\"{a}hler fourfold which is the zero locus of a general global section of $\mathcal F$, 
then $Z$ is either of Beauville--Donagi type or of Debarre--Voisin type, 
that is, $(G/P, \mathcal F)$ is either $(\Gr(2, 6), S^3 \mathcal U^*)$ or $(\Gr(6, 10), \Exterior^3 \mathcal U^*)$ up to natural identifications of $\Gr(k,n)$ with $\Gr(n-k,n)$.
\end{corollary} 

In order to prove Theorem~\ref{hyperkahler fourfold}, we give a classification of the pairs $(G/P, \mathcal F)$ such that the zero locus of a general global section of $\mathcal F$ is a fourfold with trivial canonical bundle
when $G$ is a simple Lie group of exceptional type. 
A similar analysis has been parallelly obtained in~\cite{BenedettiThesis}, with the usage of \texttt{Macaulay2}.
For Grassmannians of classical Lie type, Benedetti classified such $d$-folds for $2 \leq d \leq 4$ in~\cite{Ben}. 
In~\cite{IIM19}, Inoue, Ito and Miura also gave a similar classification and geometric descriptions of Calabi--Yau $3$-folds for Grassmannians.

\begin{theorem} 
\label{classification}
Let $G/P$ be a rational homogeneous variety of Picard number one, and let $\mathcal F$ be a completely reducible, globally generated, equivariant vector bundle over $G/P$.
If $G$ is a simple Lie group of exceptional type and the zero locus $Z$ of a general global section of $\mathcal F$ is a fourfold with trivial canonical bundle, 
then the only possible pairs $(G/P, \mathcal F)$ are listed in Table~\ref{table1} up to natural identifications. 
\end{theorem}
\begin{table}[H]
\renewcommand{\arraystretch}{1.2}
\begin{tabular}{c c c c c c c c | l}
		\toprule
		No. & $G/P$ & $\dim (G/P)$ & $\iota(G/P)$ & $\mathcal F$ & $h^{0,2}$ & $h^{1,1}$ & $h^{1,3}$ & \\
		\midrule 
$1$	&	$E_6/P_1$	&	$16$  &	$12$ 	&	$\mathcal O(1)^{\oplus 12}$ & $0$ & $1$ & $102$ 
& Proposition~\ref{E6/P1}
\\
\midrule 
$2$	&	$E_6/P_2$	&	$21$  &	$11$ 	&	$\mathcal E_{\varpi_1}^{\oplus 2} \oplus \mathcal O(1)^{\oplus 5}$  & $0$ &  $1$ & $87$ 
& \multirow{3}{*}{Proposition~\ref{E6/P2}}
\\
$2'$	&	$E_6/P_2$	&	$21$  &	$11$ 	&	$\mathcal E_{\varpi_1} \oplus \mathcal E_{\varpi_6} \oplus \mathcal O(1)^{\oplus 5}$  & $0$ &  $1$ & $87$
\\
$2''$	&	$E_6/P_2$	&	$21$  &	$11$ 	&	$\mathcal E_{\varpi_6}^{\oplus 2} \oplus \mathcal O(1)^{\oplus 5}$  & $0$ &  $1$ & $87$
\\
\midrule 
$3$	&	$E_6/P_3$	&	$25$  &	$9$ 	&	$\mathcal E_{\varpi_1}^{\oplus 3} \oplus \mathcal E_{\varpi_6}^{\oplus 3}$  & $0$ &  $1$ & $48$
& \multirow{2}{*}{Proposition~\ref{E6/P3}}
\\
$4$	&	$E_6/P_3$	&	$25$  &	$9$ 	&	$\mathcal E_{\varpi_6}^{\oplus 4} \oplus \mathcal O(1)$  & $0$ &  $1$ & $72$
\\
\midrule 
$5$	&	$E_7/P_1$	&	$33$  &	$17$ 	&	$\mathcal E_{\varpi_7}^{\oplus 2} \oplus \mathcal O(1)^{\oplus 5}$  & $0$ &  $1$ & $87$
& Proposition~\ref{E7/P1}
\\
\midrule 
$6$	&	$F_4/P_1$	&	$15$  &	$8$ 	&	$\mathcal E_{\varpi_4} \oplus \mathcal O(1)^{\oplus 5}$  & $0$ &  $1$ & $86$ 
& Proposition~\ref{F4/P1} 
\\
\midrule 
$1'$	&	$F_4/P_4$	&	$15$  &	$11$ 	&	$\mathcal O(1)^{\oplus 11}$  & $0$ &  $1$ & $102$
& \multirow{2}{*}{Proposition~\ref{F4/P4}}
\\
$7$	&	$F_4/P_4$	&	$15$  &	$11$ 	&	$\mathcal E_{\varpi_1} \oplus \mathcal O(1)^{\oplus 4}$  & $0$ &  $1$ & $87$
\\
\midrule 
$8$	&	$G_2/P_1$	&	$5$  &	$5$ 	&	$\mathcal O(5)$  & $0$ &  $1$ & $356$
& \multirow{2}{*}{Proposition~\ref{G2/P1}}
\\
$9$	&	$G_2/P_2$	&	$5$  &	$3$ 	&	$\mathcal O(3)$  & $0$ &  $1$ & $258$
\\
		\bottomrule
\end{tabular}
\caption{Calabi--Yau $4$-folds in exceptional homogeneous varieties of Picard number one.}
\label{table1}
\end{table}
In Table~\ref{table1}, $\iota(G/P)$ means the Fano index of a rational homogeneous variety $G/P$, 
and $\mathcal E_{\lambda}$ is the irreducible equivariant vector bundle associated to a $P$-dominant weight $\lambda$ (see Section~\ref{subsec_equiv_vb_on_hom_varieties} for details).

For the maximal parabolic subgroup $P_4$ associated to a simple root $\alpha_4$ of $F_4$, 
the rational homogeneous variety $F_4/P_4$ is a general hyperplane section of the Cayley plane $E_6/P_1 \subset \mathbb P^{26}$ (see~\cite[Section~6.3]{LM}).
It follows from that the 27-dimensional fundamental $E_6$-module $V_{E_6}(\varpi_1)$ may be regarded as an $F_4$-module using the embedding of $F_4$ in $E_6$
and it decomposes as $V_{F_4}(\varpi_4) \oplus V_{F_4}(0)$ (see~\cite[Proposition~13.32]{Carter}). 
Hence, the Calabi--Yau fourfold in No. $1'$ is the same as No. $1$ in Table \ref{table1}. 
The outer automorphism of~$E_6$ induces a bundle isomorphism between $\mathcal E_{\varpi_1}$ and $\mathcal E_{\varpi_6}$; 
hence the Calabi--Yau fourfolds in No. $2', 2''$ are the same as No. $2$ in Table \ref{table1}. 
Similarly, we can identify $E_6/P_1$ with $E_6/P_6$ (respectively, $E_6/P_3$ with $E_6/P_5$) by the projective equivalences induced from the outer automorphism of $E_6$. 

As a direct consequence of the classification process, we also 
get the classification of all possible pairs $(G/P, \mathcal F)$ such that the zero locus $Z$ of a general global section of $\mathcal F$ is a threefold with trivial canonical bundle.
\begin{proposition} 
Let $G/P$ be a rational homogeneous variety of Picard number one, and let $\mathcal F$ be a completely reducible, globally generated, equivariant vector bundle over $G/P$.
If $G$ is a simple Lie group of exceptional type and the zero locus $Z$ of a general global section of $\mathcal F$ is a threefold with trivial canonical bundle, 
then the only possible pairs $(G/P, \mathcal F)$ are listed in Table~\ref{table2} up to natural identifications. 
\end{proposition}
\begin{table}[H]
\renewcommand{\arraystretch}{1.2}
\begin{tabular}{c c c c c c c c c}
		\toprule
		No. & $G/P$ & $\dim (G/P)$ & $\iota(G/P)$ & $\mathcal F$ &	$h^{0,1}$, $h^{0,2}$	&	$h^{0,0}$, $h^{0,3}$, $h^{1,1}$	&	 $h^{1,2}$ 	&	 Euler characteristic $\chi$ \\
		\midrule 
$1$	&	$E_6/P_3$	&	$25$  &	$9$ 	&	$\mathcal E_{\varpi_1} \oplus \mathcal E_{\varpi_6}^{\oplus 4}$	&	$0$	&	$1$	&	$31$	&	$-60$
\\
$2$	&	$G_2/P_1$	&	$5$  &	$5$ 	&	$\mathcal O(1) \oplus \mathcal O(4)$ 		& 	$0$	&	$1$	&	$89$	&	$-176$
\\
$3$	&	$G_2/P_1$	&	$5$  &	$5$ 	&	$\mathcal O(2) \oplus \mathcal O(3)$ 		& 	$0$	&	$1$	&	$73$	&	$-144$
\\
$4$	&	$G_2/P_1$	&	$5$  &	$5$ 	&	$\mathcal E_{\varpi_2} \otimes \mathcal O(1)$ 		& 	$0$	&	$1$	&	$50$	&	$-98$
\\
$5$	&	$G_2/P_2$	&	$5$  &	$3$ 	&	$\mathcal O(1) \oplus \mathcal O(2)$ 		&	$0$	&	$1$	&	$61$	&	$-120$
\\
$6$	&	$G_2/P_2$	&	$5$  &	$3$ 	&	$\mathcal E_{\varpi_1} \otimes \mathcal O(1)$ 		&	$0$	&	$1$	&	$50$	&	$-98$
\\
		\bottomrule
\end{tabular}
\caption{Calabi--Yau $3$-folds in exceptional homogeneous varieties of Picard number one.}
\label{table2}
\end{table}

We notice that the result already obtained in \cite{IMOU} with the aid of a computer program \texttt{Mathematica} and the Hodge numbers for No. 1, No.~4, No.~5, No.~6 were computed in that paper. 
As $G_2/P_1$ is isomorphic to the hyperquadric $Q^5 \subset \mathbb{P}^6$, the threefolds given in No.~2 and No.~3 are complete intersection Calabi--Yau threefolds in $\mathbb{P}^6$, and their Hodge numbers are well-known (for example, see \cite{GHL89}). 

The paper is organized as follows. 
In Section~\ref{section_Hyperkahler}, we recall well-known properties of hyperk\"{a}hler manifolds and equivariant vector bundles over rational homogeneous varieties. 
In Section~\ref{sec_classification_fourfolds}, we list up all possible pairs $(G/P, \mathcal F)$ such that the zero locus $Z$ of a general global section of $\mathcal F$ is a fourfold with trivial canonical bundle. 
In Section~\ref{section_Hodge_numbers}, the Hodge numbers $h^{2, 0}(Z)$ and $h^{1, 3}(Z)$ of $Z$ are computed to prove Theorem~\ref{hyperkahler fourfold} and Theorem~\ref{classification}. 
Throughout this paper, we work over the complex number field $\mathbb C$.

\subsection*{Acknowledgments} 
The authors would like to express thank Professor Jaehyun Hong for helpful discussions. 
The authors also thank Vladimiro Benedetti and Makoto Miura for noticing some omission in Tables and informing references \cite{BenedettiThesis} and \cite{IMOU}, respectively.  
The first author was supported by the Institute for Basic Science (IBS-R003-D1), by the National Research Foundation of Korea (NRF) grant funded by the Korea government (MSIT) (RS-2023-00239947), and by POSCO Science Fellowship of POSCO TJ Park Foundation.  
The second author was supported by the Institute for Basic Science (IBS-R003-D1) and by the National Research Foundation of Korea (NRF) grant funded by the Korea government (MSIT) (NRF-2019R1A2C3010487, NRF-2021R1C1C2092610). 
He was also supported by Learning \& Academic research institution for Master’s·PhD students, and Postdocs (LAMP) Program of the National Research Foundation of Korea (NRF) grant funded by the Ministry of Education (RS-2023-00301974).

\section{Hyperk\"{a}hler manifolds and equivariant vector bundles} 
\label{section_Hyperkahler}

\subsection{Hyperk\"{a}hler manifolds and irreducible holomorphic symplectic manifolds} 

In this section, we recall notion of holomorphic symplectic forms and holomorphic symplectic manifolds from~\cite{Besse87} and their properties.
Let $\mathbb H$ be the associative algebra of quaternions with the imaginary units $i, j, k$. 
The symplectic group $\Sp(m)$ is defined as the group of $m \times m$ matrices $A$ over $\mathbb H$ satisfying $A\overline{A}^t=I$, 
where $x \mapsto \overline{x}$ is the conjugation given by $\overline{x}=x_0 - x_1 i - x_2 j - x_3 k$ for $x=x_0 + x_1 i + x_2 j + x_3 k \in \mathbb H$. 

\begin{definition} \label{Hyperkaehler manifolds}
A $4m$-dimensional Riemannian manifold $(M, g)$ is called \emph{hyperk\"{a}hler} if its holonomy group~$\Hol(M,g)$ is isomorphic to the symplectic group $\Sp(m)$, that is, $\Hol(M, g) \cong \Sp(m)$. 
\end{definition}

A hyperk\"{a}hler manifold $(M,g)$ is naturally equipped with three (almost) complex structures $J_1$, $J_2$, $J_3$ 
such that $J_1 \circ J_2 = J_3$ and $\nabla J_{\ell}= 0$ for $\ell =1, 2, 3$, 
where $\nabla$ is the Levi-Civita connection of $g$, i.e., $g$ is a K\"{a}hler metric with respect to each of these complex structures.
From the corresponding K\"{a}hler forms $\omega_1, \omega_2, \omega_3$, 
we get a complex $2$-form $\omega_2 + \sqrt{-1} \omega_3$ with respect to $J_1$ which makes $(M, J_1)$ a \emph{holomorphic symplectic manifold}. 

\begin{definition} \label{Irreducible holomorphic symplectic manifolds}
Let $(X, J)$ be a complex manifold of (complex) dimension $2m$. 
A closed holomorphic $2$-form~$\omega$ is called \emph{holomorphic symplectic} if the $m$th wedge product $\omega^m$ is nonzero at every point. 
An \emph{irreducible holomorphic symplectic manifold} is a simply-connected compact complex K\"{a}hler manifold $(X, J)$ 
such that the space of global holomorphic $2$-forms $H^0(X, \Omega^2_X)$ is generated by a holomorphic symplectic $2$-form.
\end{definition}

\begin{remark}
Note that $\omega^m$ is a nonvanishing holomorphic global section of the canonical line bundle $K_X= \Exterior^{2m} T^* X$. 
Thus, holomorphic symplectic manifolds have trivial canonical bundle, and admit Ricci-flat metrics by Yau's theorem~\cite{Yau78}. 
\end{remark}

\begin{definition}
A compact K\"{a}hler manifold of complex dimension $n$ is called to be \emph{Calabi--Yau} if its canonical bundle is trivial, or equivalently, it admits a (Ricci-flat) K\"{a}hler metric with holonomy group contained in $\SU(n)$.
\end{definition}

We can relate these differential geometric concepts to algebraic geometric objects such as Dolbeault cohomology and sheaf cohomology due to Beauville.

\begin{proposition}[{\cite[Proposition~4]{Beauville83}}] \label{Beauville's theorem}
If $(M, J, g)$ is a compact connected hyperk\"{a}hler manifold, then $M$ is simply-connected and admits a unique holomorphic symplectic structure $\omega$ \textup{(}up to scaling factor\textup{)}. 
Conversely, if $(M, J, \omega)$ is an irreducible holomorphic symplectic manifold, then it admits a hyperk\"{a}hler metric. 
\end{proposition}

Consequently, we use the names `irreducible holomorphic symplectic manifolds' and `compact hyperk\"{a}hler manifolds' interchangeably. 

\begin{corollary} \label{criterion}
Let $(M, J)$ be a compact connected K\"{a}hler fourfold with trivial canonical bundle. 
Then $M$ carries a hyperk\"{a}hler metric $g$ if and only if $h^2(M, \mathcal O_M)=1$ for the structure sheaf $\mathcal O_M$ of $M$. 
\end{corollary}

\begin{proof}
Suppose that $M$ admits a hyperk\"{a}hler metric. 
Then $M$ is an irreducible holomorphic symplectic manifold by Proposition~\ref{Beauville's theorem}, hence $h^0(M, \Omega_M^2) = \dim H^0(M, \Omega_M^2) = 1$. 
Since the Dolbeault isomorphism theorem says that $h^{2,0}(M) = h^0(M, \Omega_M^2)$ and $h^{0,2}(M) = h^2(M, \mathcal O_M)$, 
we have  $h^2(M, \mathcal O_M)=1$ from the conjugation $h^{2,0}(M) = h^{0,2}(M)$. 

Conversely, if $M$ is connected and $h^2(M, \mathcal O_M)=1$, then the Euler--Poincar\'e characteristic $\chi(\mathcal O_M)$ is equal to 3 so that $M$ is simply-connected 
by \cite[Proposition 3.16]{Ben} using the Beauville--Bogomolov decomposition theorem \cite[Th\'eor\`eme~1]{Beauville83} for compact K\"{a}hler manifolds with trivial canonical bundle. 
As $M$ is a compact simply-connected K\"{a}hler fourfold with trivial canonical bundle, the Bogomolov decomposition theorem in \cite{Bogomolov} implies that $M$ is either a product of two K3 surfaces, or a Calabi--Yau fourfold with $h^0(M, \Omega_M^2)=0$, or an irreducible holomorphic symplectic fourfold. 
From the assumption $h^0(M, \Omega_M^2)=h^2(M, \mathcal O_M)=1$, the manifold $M$ is irreducible holomorphic symplectic and $H^0(M, \Omega^2_M)$ is generated by a holomorphic symplectic $2$-form. 
Thus, the result follows from Proposition~\ref{Beauville's theorem}.
\end{proof}

\begin{proposition}\label{direct sum of line bundles} 
Let $Z$ be a fourfold with trivial canonical bundle which is a general complete intersection with respect to a  direct sum of line bundles over a smooth Fano variety $X$. 
Then $Z$ is not hyperk\"{a}hler but Calabi--Yau. 
\end{proposition}

\begin{proof}
Suppose that $X$ has dimension $n$ and Fano index $\iota$. 
By the assumption, $Z$ is the zero locus of a general global section of $\mathcal F = \mathcal O_{X}(a_j)^{\oplus (n-4)}$ with $\sum_{j=1}^{n-4} a_j=\iota$ and $a_j > 0$. 
It is immediately checked that $H^p(X, \Exterior^q \mathcal F^*)$ are nontrivial only for $p=n, q=n-4$ by the Kodaira vanishing theorem. 
Using the Serre duality, 
$H^{n}(X, \Exterior^{n-4} \mathcal F^*) = H^0(X, \mathcal O_X)^* = \mathbb C$. 
Then, from the Koszul complex
$$0\to \Exterior^{n-4} \mathcal F^* \to \Exterior^{n-5} \mathcal F^* \to \cdots \to \Exterior^2 \mathcal F^* \to \mathcal F^* \to \mathcal O_X \to \mathcal O_Z \to 0,$$
we have $h^4(Z, \mathcal O_Z)=1$ and $h^2(Z, \mathcal O_Z)=0$ 
which means that $Z$ is not hyperk\"{a}hler by Corollary~\ref{criterion}.
\end{proof}

In general, complete intersections of dimension greater than $2$ are never irreducible holomorphic symplectic, 
which is one reason why it is extremely hard to construct irreducible holomorphic symplectic manifolds.

\subsection{Equivariant vector bundles over homogeneous varieties} 
\label{subsec_equiv_vb_on_hom_varieties}

Let $G$ be a simply-connected complex semisimple Lie group and $P \subset G$ be a parabolic subgroup. 
We will follow the notations in~\cite{FH} for the basics on the representation theory of a Lie algebra.
For an integral dominant weight $\lambda$ with respect to $P$, 
we have an irreducible representation $V_P(\lambda)$ of $P$ with the highest weight $\lambda$, and 
denote by $\mathcal E_{\lambda}$ the corresponding \emph{irreducible equivariant vector bundle} $G\times_P V_P(\lambda)^*$ over $G/P$:
$$\mathcal E_{\lambda}\coloneqq G \times_P V_P(\lambda)^* = (G \times V_P(\lambda)^*)/P,$$
where the equivalence relation is given by $(g, v) \sim (gp, p^{-1} . v)$ for $p \in P$.

\begin{theorem}[Borel--Weil--Bott theorem~\cite{Bott57}] 
\label{BBW}
Let $G$ be a simply-connected complex semisimple Lie group and $P \subset G$ be a parabolic subgroup. 
Let $\rho$ denote the sum of fundamental weights of $G$. 
For an integral dominant weight $\lambda$ with respect to $P$, the following holds. 
\begin{itemize}
\item If a weight $\lambda+\rho$ is singular, that is, it is orthogonal to some \textup{(}positive\textup{)} root of $G$, 
then the cohomology groups $H^i(G/P, \mathcal E_{\lambda})$ vanish for all $i$. 
\item Otherwise, $\lambda+\rho$ is regular, that is, it lies in the interior of some Weyl chamber, 
then 
\[
H^{\ell(w)}(G/P, \mathcal E_{\lambda})=V_G(w(\lambda+\rho)-\rho)^\ast
\] 
and any other cohomology vanishes. 
Here, $w\in W$ is a unique element of the Weyl group of $G$ such that $w(\lambda+\rho)$ is regular dominant, 
and $\ell(w)$ means the length of $w \in W$, that is, $\ell(w)$ is the minimum integer such that $w$ can be expressed as a product of $\ell(w)$ simple reflections.  
\end{itemize}
\end{theorem}

Unfortunately, since $P$ is not reductive in general, its representations are difficult to study. 
Nevertheless, it is often possible to get results about an equivariant vector bundle from irreducible bundles by considering a filtration by $P$-submodules.
If $V$ is a $P$-module, it has a filtration $0 \subset V_t \subset \cdots \subset V_1 \subset V_0 = V$ such that $V_{i}/V_{i+1}$ is an irreducible $P$-module.
Moreover, we can consider $V$ as an $L$-module by the restriction for the reductive part $L$ of a Levi decomposition of $P=L U$, 
and define a new $P$-module structure by extending trivially to the unipotent radical $U$.
The associated graded module $\gr(V)$ is a direct sum of irreducible $P$-modules. 
By a series of extensions, $V$ can be reconstructed from the summands of $\gr(V)$. 

\begin{example}[Beauville--Donagi~\cite{BD}]
\label{Beauville--Donagi}
Let $Z$ be the zero locus of a general global section of the third symmetric power $S^3 \mathcal U^*$ of the dual universal subbundle of the Grassmannian $\Gr(2, 6)$. 
Recall that the dimension of $\Gr(2, 6)$ is 8 and $K_{\Gr(2, 6)} = \mathcal O(-6)$. 
Since the rank of $S^3 \mathcal U^*$ is 4 and $\det (S^3 \mathcal U^*) \cong \mathcal O(6)$ (see Proposition~\ref{dex of classical types}), 
$Z$ is a fourfold with trivial canonical bundle by the adjunction formula. 
Then we have the Koszul complex associated to a general section $s$ of the equivariant vector bundle $\mathcal F = S^3 \mathcal U^*$: 
$$0\to \Exterior^4 \mathcal F^* \to \Exterior^3 \mathcal F^* \to  \Exterior^2 \mathcal F^* \to \mathcal F^* \to \mathcal O_{\Gr(2,6)} \to \mathcal O_Z \to 0.$$
Using the Littlewood--Richardson rule (see~\cite[Section~2.3]{Weyman} for details), we obtain 
$$0\to \mathcal O(-6) \to S^3 \mathcal U(-3) \to S^4 \mathcal U(-1) \oplus \mathcal O(-3) \to S^3 \mathcal U \to \mathcal O_{\Gr(2,6)} \to \mathcal O_Z \to 0.$$
By the Borel--Weil--Bott theorem, 
$H^8(\Gr(2,6), \mathcal O(-6))=H^4(\Gr(2,6), S^4 \mathcal U(-1))=\mathbb C$ and the other non-trivial cohomologies vanish. 
For example, $S^3 \mathcal U(-3) = \mathcal E_{3\varpi_1 -6 \varpi_2}$ and the weight
\[
3\varpi_1 -6 \varpi_2 + \rho = 4\varpi_1 - 5\varpi_2 + \varpi_3 + \varpi_4 + \varpi_5
\]  
is singular, 
where $\varpi_1, \dots, \varpi_5$ are the fundamental weights of $\SL(6, \mathbb C)$. 
As a straightforward application of the Borel--Weil--Bott theorem, 
we see $H^{i}(\Gr(2,6), S^3 \mathcal U(-3))=0$ for all $i$.
Therefore, $H^0(Z, \Omega^2_Z) \cong H^2(Z, \mathcal O_Z)=\mathbb C$ and $Z$ is a hyperk\"{a}hler fourfold by Corollary~\ref{criterion}.
\end{example}

\begin{remark}
Similarly, we can prove that the zero locus of a general global section of the third exterior power $\Exterior^3 \mathcal U^*$ of the dual universal subbundle of the Grassmannian $\Gr(6, 10)$ is also a hyperk\"{a}hler fourfold (see~\cite[Remark~2.6]{DV}). 
\end{remark}

\section{Classification of fourfolds with trivial canonical bundle}
\label{sec_classification_fourfolds}
Let $G/P$ be an exceptional homogeneous variety of Picard number one, and $\mathcal F$ be a completely reducible, globally generated, equivariant vector bundle over $G/P$. 
Here, we say that a vector bundle $\mathcal F$ is completely reducible if it can be expressed as a direct sum of irreducible vector bundles. 
In this section, we list up all possible pairs $(G/P, \mathcal F)$ such that the zero locus $Z$ of a general global section of $\mathcal F$ is a fourfold with trivial canonical bundle based on the method done by K\"{u}chle~\cite{Kuchle} and Benedetti~\cite{Ben}. 

In what follows, we fix an ordering on the simple roots as in Table~\ref{Dynkin}; our conventions agree with that in~\cite{FH}, which is called the Bourbaki ordering for the simple roots. 
Moreover, for a parabolic subgroup~$P$, a weight $\lambda$ is called \emph{$P$-dominant} if $\langle \lambda, \alpha \rangle \geq 0$ for all positive roots $\alpha$ such that $\mathfrak{g}_{\alpha} \subset \mathfrak{l}$,
where $\mathfrak{g}_{\alpha}$ is the root space corresponding to $\alpha$ and $\mathfrak{l}$ is the Lie algebra of the reductive part $L$ in a Levi decomposition of $P$. 
	\begin{table}[b]
	\bgroup
	\def\arraystretch{1.5}
		\begin{tabular}{ll|ll }
			\toprule
			$\Phi$ & Dynkin diagram & $\Phi$ & Dynkin diagram  \\
			\midrule
			$E_6$ & 
			\begin{tikzpicture}[scale=.5, baseline=-0.5ex]
\node[state, label=below:{1}] (1) {};
\node[state, label=below:{3}] (3) [right = of 1] {};
\node[state, label=below:{4}] (4) [right = of 3] {};
\node[state, label=right:{2}] (2) [above =of 4] {};
\node[state, label=below:{5}] (5) [right = of 4] {};
\node[state, label=below:{6}] (6) [right = of 5] {};

\draw (1)--(3)--(4)--(5)--(6);
\draw (2)--(4);
			\end{tikzpicture}
			
			&
			$E_7$& 
			\begin{tikzpicture}[scale=.5, baseline=-0.5ex]
\node[state, label=below:{1}] (1) {};
\node[state, label=below:{3}] (3) [right = of 1] {};
\node[state, label=below:{4}] (4) [right = of 3] {};
\node[state, label=right:{2}] (2) [above =of 4] {};
\node[state, label=below:{5}] (5) [right = of 4] {};
\node[state, label=below:{6}] (6) [right = of 5] {};
\node[state, label=below:{7}] (7) [right = of 6] {};

\draw (1)--(3)--(4)--(5)--(6)--(7);
\draw (2)--(4);
			\end{tikzpicture}
			
			\\
			$E_8$ & 
			\begin{tikzpicture}[scale=.5, baseline=-0.5ex]
\node[state, label=below:{1}] (1) {};
\node[state, label=below:{3}] (3) [right = of 1] {};
\node[state, label=below:{4}] (4) [right = of 3] {};
\node[state, label=right:{2}] (2) [above =of 4] {};
\node[state, label=below:{5}] (5) [right = of 4] {};
\node[state, label=below:{6}] (6) [right = of 5] {};
\node[state, label=below:{7}] (7) [right = of 6] {};
\node[state, label=below:{8}] (8) [right = of 7] {};

\draw (1)--(3)--(4)--(5)--(6)--(7)--(8);
\draw (2)--(4);
			\end{tikzpicture} \\
			$F_4$
			&
			\begin{tikzpicture}[scale = .5, baseline=-0.5ex]
			\node[state, label=below:{1}] (1) {};
			\node[state, label=below:{2}] (2) [right = of 1] {};
			\node[state, label=below:{3}] (3) [right = of 2] {};
			\node[state, label=below:{4}] (4) [right =of 3] {};
			
			\draw (1)--(2)
			(3)--(4);
			\draw[double line] (2)-- node{$>$} (3);
			\end{tikzpicture} 
			&
			$G_2$
			&
			\begin{tikzpicture}[scale =.5, baseline=-0.5ex]

			\node[state, label=below:{1}] (1) {};
			\node[state, label=below:{2}] (2) [right = of 1] {};
			
			\draw[triple line] (1)-- node{$<$} (2);
			\draw (1)--(2);
			
			\end{tikzpicture}
			\\
			\bottomrule
		\end{tabular}
		\captionof{table}{Dynkin diagrams of exceptional Lie groups.} 
		\label{Dynkin}
		\egroup
	\end{table}

\begin{definition} \label{dex}
Let $P_k \subset G$ be the $k$th maximal parabolic subgroup and $\mathcal O(1)$ be the ample generator of the Picard group of $G/P_k$ giving the minimal embedding in the projective space $\mathbb P(V_G(\varpi_k))$, that is, $\mathcal O(1) =\mathcal E_{\varpi_k}$.  
Here, $\varpi_k$ is the $k$th fundamental weight. 
The \emph{dex} of a vector bundle $\mathcal F$ over $G/P_k$ is an integer~$\dex(\mathcal F)$ defined by $\det(\mathcal F) = \mathcal O(1)^{\dex(\mathcal F)}$.
Similarly, the \emph{dex} of an integral $P$-dominant weight $\lambda$ is defined as $\dex(\lambda) = \dex(\mathcal E_{\lambda})$. 
\end{definition}

The following proposition is essentially due to Benedetti in~\cite[Sections~3 and~4]{Ben} but we provide its proof for reader's convenience.
\begin{proposition}
\label{dex of classical types} 
Let $G$ be a semisimple Lie group of rank $r$, and $P_k$ be the $k$th maximal parabolic subgroup.  
\begin{enumerate}
\item[{\rm (1)}] \textup{(}$A$ type\textup{)} For a Grassmannian $\SL(r+1, \mathbb C)/P_k = \Gr(k, r+1)$, and for $\lambda=\sum_{i=1}^r \lambda_i \varpi_i$, we have 
\[
\dex(\lambda) = \Big(\frac{\sum_{j=1}^k \sum_{i=j}^r \lambda_i}{k} - \frac{\sum_{j=k+1}^r \sum_{i=j}^r \lambda_i}{r+1-k} \Big) \rank(\mathcal E_{\lambda}).
\]

\item[{\rm (2)}]  \textup{(}$C$ type\textup{)} For a symplectic Grassmannian $\Sp(2r, \mathbb C)/P_k = \IGr(k,2r)$, and for $\lambda=\sum_{i=1}^r \lambda_i \varpi_i$, we have 
\[
\dex(\lambda) = \frac{\sum_{j=1}^k \sum_{i=j}^r \lambda_i}{k} \cdot \rank(\mathcal E_{\lambda}).
\]

\item[{\rm (3)}]  \textup{(}$B$, $D$ types\textup{)} Two homogeneous varieties $\SO(2r, \mathbb C)/P_r$ and $\SO(2r-1, \mathbb C)/P_{r-1}$ are isomorphic, 
called a \emph{spinor variety} $\mathbb S_r$ of dimension $\frac{r(r-1)}{2}$. 
For a spinor variety $\mathbb S_r = \SO(2r, \mathbb C)/P_r$, 
and for $\lambda=\sum_{i=1}^r \lambda_i \varpi_i$, we have 
\[
\dex(\lambda) = 2 \cdot \frac{\sum_{j=1}^r (\sum_{i=j}^{r-2} \lambda_i +\frac{\lambda_{r-1}+\lambda_{r}}{2})}{r} \cdot \rank(\mathcal E_{\lambda}).
\]
\end{enumerate}
\end{proposition}

\begin{proof}
\noindent (1) The first formula is the same as in~\cite[Lemma 3.4]{Kuchle} and~\cite[Lemma 3.8]{Ben}. 
Indeed, it follows from the facts that $\dex(\lambda) \varpi_k$ is the sum of all weights of the irreducible $P$-module with highest weight $\lambda$ and 
the set of weights is invariant under the action of the Weyl group of $P$. 

\smallskip 
\noindent (2) This is explained in~\cite[Remark~3.9 and Section~4.1]{Ben}. 
Since the semisimple part of $P_k$ is isomorphic to $\SL(k, \mathbb C) \times \Sp(2r-2k, \mathbb C)$, 
its Weyl group $W$ is isomorphic to $\mathfrak S_k \times (\mathfrak S_{r-k} \ltimes \mathbb Z_2^{r-k})$, 
where $\mathfrak S_n$ stands for the symmetric group on $n$ letters. 
Since the sum of weights given by the signed symmetric group $\mathfrak S_{r-k} \ltimes \mathbb Z_2^{r-k}$ is equal to zero, we get the result. 

\smallskip
\noindent (3) Let $\OGr(r, 2r)$ be the variety parametrizing isotropic $r$-subspaces in $2r$-dimensional vector space equipped with a nondegenerate symmetric bilinear form.
Then $\OGr(r, 2r)$ has two connected components and these two components are in fact indistinguishable as embedded varieties (for example, see~\cite[Section~23.3]{FH}).
The spinor variety $\mathbb S_r$ is one of two components of $\OGr(r, 2r)$, 
and has the natural embedding $\mathbb S_r \subset \Gr(r, 2r) \subset \mathbb P(\Exterior^r \mathbb C^{2r})$.
However, the restriction of the Pl\"{u}cker line bundle over $\Gr(r, 2r)$ to $\mathbb S_r$ is divisible by two, 
and the square root of the Pl\"{u}cker line bundle gives an embedding into the projectivization of a spin representation of $\SO(2r, \mathbb C)$. 
Because the Weyl group of $P_r$ is isomorphic to $\mathfrak S_r$, 
a similar argument holds as in (1), but we have to multiply by 2 to get the correct formula. 
\end{proof}

Let $Z=Z_{\mathcal F}$ be the zero locus of a general global section of $\mathcal F$ over a rational homogeneous variety $G/P$ of Picard number one. 
If $Z$ has the properties that $\dim Z = 4$ and $K_Z=\mathcal O_Z$, 
by the adjunction formula for $Z \subset G/P$, we have  
\begin{equation}\label{conditions}
\rank (\mathcal F) = \dim(G/P)-4 \quad \text{and} \quad \dex (\mathcal F) = \iota(G/P),
\end{equation}
where $\iota(G/P)$ means the Fano index of a rational homogeneous variety $G/P$. 
Note that all rational homogeneous varieties $G/P$ are Fano varieties (for example, see Sections II.4.2--II.4.4 of \cite{Jantzen}). 

\begin{lemma} \label{ratio} 
Let $G/P$ be a homogeneous variety of Picard number one. 
Let $\mathcal F = \bigoplus_{i=1}^m \, \mathcal F_i$ be the direct sum of irreducible equivariant vector bundles over $G/P$.
If the inequalities 
\[
\frac{\dex(\mathcal F_i)}{\rank(\mathcal F_i)} > \frac{\iota(G/P)}{\dim(G/P)-4}
\]
hold for all $i$, 
then there is no fourfold $Z$ such that it is the zero locus of a general global section of $\mathcal F$ and has trivial canonical bundle.
\end{lemma}

\begin{proof}
Assume that there is a fourfold $Z$ satisfying the given conditions. 
Since $\mathcal F$ is the direct sum of $\mathcal F_i$'s, we have $\rank (\mathcal F) = \sum_i \rank (\mathcal F_i)$ and $\dex (\mathcal F) = \sum_i \dex (\mathcal F_i)$. 
Then 
\[
\frac{\dex(\mathcal F)}{\rank(\mathcal F)} = \frac{\sum_i \dex(\mathcal F_i)}{\sum_i \rank(\mathcal F_i)} > \frac{\iota(G/P)}{\dim(G/P)-4}
\]
by the assumption. 
However, it contradicts to the conditions in~\eqref{conditions}. 
\end{proof}

From the same argument used in the proof of Lemma~\ref{ratio}, we have the following consequence directly.

\begin{proposition}\label{linear section} 
Let $G/P$ be a homogeneous variety of Picard number one.
Let $\mathcal F = \bigoplus_{i=1}^m \, \mathcal F_i$ be the direct sum of line bundles over $G/P$. 
Suppose that $\dim(G/P) = \iota(G/P) + 4$ and 
$\dex(\mathcal F_i) \geq 1$ for all irreducible summands $\mathcal F_i$ of~$\mathcal F$. 
If $Z \subset G/P$ is a fourfold with trivial canonical bundle which is the zero locus of a general global section of $\mathcal F$, 
then $Z$ is a general linear section of $G/P$, that is, $\mathcal F$ is isomorphic to $\mathcal O(1)^{\oplus (\dim(G/P)-4)}$. 
\end{proposition}

Now, we will give a classification of possible pairs $(G/P, \mathcal F)$ such that the zero locus $Z$ of a general global section of $\mathcal F$ is a fourfold with trivial canonical bundle
when $G/P$ is an exceptional homogeneous variety of Picard number one. 
This is done by a case-by-case analysis using the ratio dex/rank of irreducible equivariant vector bundles over $G/P$. 

\begin{proposition}\label{prop_3.5}
Let $G/P$ be an exceptional homogeneous variety of Picard number one.
Let $\mathcal F = \bigoplus_{i=1}^m \, \mathcal F_i$ be the direct sum of irreducible equivariant vector bundles over $G/P$. 
Let $Z \subset G/P$ be the zero locus of a general global section of $\mathcal F$. 
If $Z$ is a fourfold with trivial canonical bundle, 
then the only possible cases for $G/P$ are $E_6/P_1, E_6/P_2, E_6/P_3, E_7/P_1, F_4/P_1, F_4/P_4$, $G_2/P_1$, and $G_2/P_2$ up to natural identifications.  
\end{proposition}

\begin{proof}
It is well-known that dimensions and Fano indices of rational homogeneous varieties $G/P$ are computed from the root information 
(for example, see~\cite{Snow}). 

\smallskip 
\noindent (1) Suppose that $G$ is of type $E_6$.
Note that the dual Cayley plane $E_6/P_6 \subset \mathbb P^{26}$ is projectively equivalent to $E_6/P_1\subset \mathbb P^{26}$ 
because the highest weight $E_6$-module $V_{E_6}(\varpi_6)$ is dual to $V_{E_6}(\varpi_1)$. 
Similarly, $E_6/P_3$ is projectively equivalent to $E_6/P_5$. 
Accordingly, it is enough to consider $P = P_4$.

The rational homogeneous variety $E_6/P_4 \subset \mathbb P^{2924}$ has dimension 29 and Fano index 7. 
Since the semisimple part of $P_4$ is isomorphic to $\SL(3, \mathbb C) \times \SL(2, \mathbb C) \times \SL(3, \mathbb C)$, 
we can easily compute $\dex(\mathcal E_{\lambda})$ by Proposition~\ref{dex of classical types}(1). 
For example, $\dex (\mathcal E_{\varpi_3}) = \frac{2}{3} \times 3 = 2$. 
Since $\mathcal E_{\varpi_2+\varpi_3} = \mathcal E_{\varpi_2} \otimes \mathcal E_{\varpi_3}$, 
$\dex (\mathcal E_{\varpi_2+\varpi_3}) = \dex (\mathcal E_{\varpi_2}) \cdot \rank (\mathcal E_{\varpi_3}) + \dex (\mathcal E_{\varpi_3}) \cdot \rank (\mathcal E_{\varpi_2})
= 1 \times 3 + 2 \times 2 = 7$; see Table~\ref{table_dex_E6_P4}. 
\begin{table}[H]
\begin{tabular}{c c c c c c c c}
		\toprule
		$\lambda$ & $\varpi_1$  & $\varpi_2$  & $\varpi_3$  & $\varpi_5$  & $\varpi_6$ & $\varpi_1+\varpi_2$ & $\varpi_2+\varpi_3$
\\
		\midrule 
	$\rank (\mathcal E_{\lambda})$	&	$3$	&	$2$  &	$3$ 	&	$3$ 	& 	$3$	& 	$6$	& $6$
\\
	$\dex (\mathcal E_{\lambda})$	&	$1$	&	$1$  &	$2$ 	&	$2$ 	& 	$1$ 	& 	$5$	& 	$7$
\\
		\bottomrule
\end{tabular}
\caption{The irreducible equivariant vector bundles over $E_6/P_4$ with dex $\leq 7$.} \label{table_dex_E6_P4}
\end{table}
In general, for two vector bundles $\mathcal F_{1}$ and $\mathcal F_{2}$, the following holds
\[
\frac{\dex (\mathcal F_{1} \otimes \mathcal F_{2})}{\rank (\mathcal F_{1} \otimes \mathcal F_{2})} 
= \frac{\dex (\mathcal F_{1})}{\rank (\mathcal F_{1})} + \frac{\dex (\mathcal F_{2})}{\rank (\mathcal F_{2})}. 
\] 
Thus, we deduce that $\frac{\dex(\mathcal F_i)}{\rank(\mathcal F_i)} \geq \frac{1}{3}$ for any irreducible equivariant vector bundle $\mathcal F_{i}$ over $E_6/P_4$. 
Since 
\[
\frac{\iota(E_6/P_4)}{\dim(E_6/P_4)-4} = \frac{7}{25} < \frac{1}{3},
\] 
we cannot get a direct sum of irreducible equivariant vector bundles of which a general global section gives a fourfold with trivial canonical bundle by Lemma~\ref{ratio}.

\smallskip 
\noindent (2) Suppose that $G$ is of type $E_7$.
\begin{enumerate}[label=(\alph*)]
\item The rational homogeneous variety $E_7/P_2 \subset \mathbb P^{911}$ has dimension $42$ and Fano index $14$. 
Note that irreducible equivariant vector bundles with rank $\leq 42-4=38$ are given in Table~\ref{table_dex_E7_P2}. 
	\begin{table}[H]
\begin{tabular}{c c c c c c c c c}
		\toprule
		$\lambda$ & $\varpi_1$  & $\varpi_3$  & $\varpi_4$  & $\varpi_5$  & $\varpi_6$ & $\varpi_7$ & $2\varpi_1$ & $2\varpi_7$
\\
		\midrule 
	$\rank (\mathcal E_{\lambda})$	&	$7$	&	$21$  &	$35$ 	&	$35$ 	& 	$21$	& 	$7$	& 	$28$	& 	$28$
\\
	$\dex (\mathcal E_{\lambda})$	&	$4$	&	$24$  &	$60$ 	&	$45$ 	& 	$18$ 	& 	$3$	& 	$32$	& 	$24$
\\
		\bottomrule
\end{tabular}
\caption{The irreducible equivariant vector bundles over $E_7/P_2$ with rank $\leq 38$.} \label{table_dex_E7_P2}
\end{table}
Since 
\[
\frac{\dex(\mathcal F_i)}{\rank(\mathcal F_i)} \geq \frac{3}{7} > \frac{14}{38} = \frac{\iota(E_7/P_2)}{\dim(E_7/P_2)-4} \qquad \text{ for all }i,
\] 
there is no $Z$ with trivial canonical bundle by Lemma \ref{ratio}. 

\item The rational homogeneous variety $E_7/P_3 \subset \mathbb P^{8644}$ has dimension $47$ and Fano index $11$. 
Since the semisimple part of $P_3$ is isomorphic to $\SL(2,\C) \times \SL(6,\C)$, by Proposition~\ref{dex of classical types}, 
we know 
\[
\frac{\dex(\mathcal F_i)}{\rank(\mathcal F_i)} \geq \frac{1}{3} > \frac{11}{43} = \frac{\iota(E_7/P_3)}{\dim(E_7/P_3)-4} \qquad\text{ for all } i.
\]

\item The rational homogeneous variety $E_7/P_4 \subset \mathbb P^{365749}$ has dimension $53$ and Fano index $8$. 
Since the semisimple part of $P_4$ is isomorphic to $\SL(3,\C) \times \SL(2,\C) \times \SL(4,\C)$, by Proposition~\ref{dex of classical types},  
we know 
\[
\frac{\dex(\mathcal F_i)}{\rank(\mathcal F_i)} \geq \frac{1}{4} > \frac{8}{49} = \frac{\iota(E_7/P_4)}{\dim(E_7/P_4)-4} \qquad \text{ for all } i. 
\]

\item The rational homogeneous variety $E_7/P_5 \subset \mathbb P^{27663}$ has dimension $50$ and Fano index $10$. 
Since the semisimple part of $P_5$ is isomorphic to $\SL(5,\C) \times \SL(3,\C)$, by Proposition~\ref{dex of classical types}, 
we know 
\[
\frac{\dex(\mathcal F_i)}{\rank(\mathcal F_i)} \geq \frac{1}{3} > \frac{10}{46} = \frac{\iota(E_7/P_5)}{\dim(E_7/P_5)-4} \qquad \text{ for all } i. 
\]

\item The rational homogeneous variety $E_7/P_6 \subset \mathbb P^{1538}$ has dimension $42$ and Fano index $13$. 
From the direct computations (cf. Proposition~\ref{E6/P1}), 
\[
\frac{\dex(\mathcal F_i)}{\rank(\mathcal F_i)} \geq \frac{1}{2} > \frac{13}{38} = \frac{\iota(E_7/P_6)}{\dim(E_7/P_6)-4} \qquad \text{ for all } i.
\]

\item The rational homogeneous variety $E_7/P_7 \subset \mathbb P^{55}$ has dimension $27$ and Fano index $18$. 
The semisimple part of the maximal parabolic subgroup $P_7 \subset E_7$ is isomorphic to $E_6$. 
Since the fundamental representation $V_{E_6}(\varpi_1) \cong \mathbb C^{27}$ is a nontrivial irreducible representation of $E_6$ with minimal dimension, 
an irreducible representation of $E_6$ having dimension $\leq 27-4=23$ is only the trivial representation. 
Consequently, $\mathcal F$ is isomorphic to a direct sum of line bundles 
so that 
\[
\frac{\dex(\mathcal F_i)}{\rank(\mathcal F_i)} = \dex(\mathcal F_i) \geq 1 > \frac{18}{23} = \frac{\iota(E_7/P_7)}{\dim(E_7/P_7)-4} \qquad \text{ for all }i.
\]
\end{enumerate}

\smallskip 
\noindent (3) Suppose that $G$ is of type $E_8$.
\begin{enumerate}[label=(\alph*)]
\item The rational homogeneous variety $E_8/P_1 \subset \mathbb P^{3874}$ has dimension $78$ and Fano index $23$. 
Note that irreducible equivariant vector bundles with rank $\leq 78-4=74$ are  $\mathcal E_{\varpi_2}$, $\mathcal E_{\varpi_3}$ and $\mathcal E_{\varpi_8}$.
From the direct computations, 
\[
\frac{\dex(\mathcal F_i)}{\rank(\mathcal F_i)} \geq \frac{1}{2} > \frac{23}{74} = \frac{\iota(E_8/P_1)}{\dim(E_8/P_1)-4} \qquad \text{ for all }i.
\] 

\item The rational homogeneous variety $E_8/P_2 \subset \mathbb P^{147249}$ has dimension $92$ and Fano index $17$. 
From the direct computations, 
\[
\frac{\dex(\mathcal F_i)}{\rank(\mathcal F_i)} \geq \frac{3}{8} > \frac{17}{88} = \frac{\iota(E_8/P_2)}{\dim(E_8/P_2)-4} \qquad \text{ for all }i.
\]

\item The rational homogeneous variety $E_8/P_3 \subset \mathbb P^{6695999}$ has dimension $98$ and Fano index $13$. 
Since the semisimple part of $P_3$ is isomorphic to $\SL(2,\C) \times \SL(7,\C)$, by Proposition~\ref{dex of classical types}, 
we know 
\[
\frac{\dex(\mathcal F_i)}{\rank(\mathcal F_i)} \geq \frac{2}{7} > \frac{13}{94} = \frac{\iota(E_8/P_3)}{\dim(E_8/P_3)-4} \qquad \text{ for all }i.
\]

\item The rational homogeneous variety $E_8/P_4 \subset \mathbb P^{6899079263}$ has dimension $106$ and Fano index $9$. 
Since the semisimple part of $P_4$ is isomorphic to $\SL(3,\C) \times \SL(2,\C) \times \SL(5,\C)$, by Proposition~\ref{dex of classical types}, 
we know 
\[
\frac{\dex(\mathcal F_i)}{\rank(\mathcal F_i)} \geq \frac{1}{5} > \frac{9}{102} = \frac{\iota(E_8/P_4)}{\dim(E_8/P_4)-4} \qquad \text{ for all }i.
\] 

\item The rational homogeneous variety $E_8/P_5 \subset \mathbb P^{146325269}$ has dimension $104$ and Fano index $11$. 
Since the semisimple part of $P_5$ is isomorphic to $\SL(5,\C) \times \SL(4,\C)$, by Proposition~\ref{dex of classical types}, 
we know 
\[
\frac{\dex(\mathcal F_i)}{\rank(\mathcal F_i)} \geq \frac{1}{4} > \frac{11}{100} = \frac{\iota(E_8/P_5)}{\dim(E_8/P_5)-4} \qquad \text{ for all }i.
\] 

\item The rational homogeneous variety $E_8/P_6 \subset \mathbb P^{2450239}$ has dimension $97$ and Fano index $14$. 
Since the semisimple part of $P_6$ is isomorphic to $\Spin(10, \C) \times \SL(3, \C)$, from Proposition~\ref{dex of classical types}  
and the direct computations, 
\[
\frac{\dex(\mathcal F_i)}{\rank(\mathcal F_i)} \geq \frac{1}{3} > \frac{14}{93} = \frac{\iota(E_8/P_6)}{\dim(E_8/P_6)-4} \qquad \text{ for all }i.
\] 

\item The rational homogeneous variety $E_8/P_7 \subset \mathbb P^{30379}$ has dimension $83$ and Fano index $19$. 
From the direct computations, 
\[
\frac{\dex(\mathcal F_i)}{\rank(\mathcal F_i)} > \frac{19}{79} = \frac{\iota(E_8/P_7)}{\dim(E_8/P_7)-4}
\qquad \text{ for all }i.
\] 

\item The rational homogeneous variety $E_8/P_8 \subset \mathbb P^{247}$ has dimension $57$ and Fano index $29$. 
The semisimple part of the maximal parabolic subgroup $P_8 \subset E_8$ is isomorphic to $E_7$. 
Since the fundamental representation $V_{E_7}(\varpi_1) \cong \mathbb C^{56}$ is a nontrivial irreducible representation of $E_7$ with minimal dimension,  
an irreducible representation of $E_7$ having dimension $\leq 57-4=53$ is only the trivial representation. 
Consequently, $\mathcal F$ is isomorphic to a direct sum of line bundles. 
\end{enumerate}

\smallskip
\noindent (4) Suppose that $G$ is of type $F_4$.
\begin{enumerate}[label=(\alph*)]
\item The rational homogeneous variety $F_4/P_2 \subset \mathbb P^{1273}$ has dimension $20$ and Fano index $5$. 
Since the semisimple part of $P_2$ is isomorphic to $\SL(2,\C) \times \SL(3,\C)$, by Proposition~\ref{dex of classical types}, 
we know 
\[
\frac{\dex(\mathcal F_i)}{\rank(\mathcal F_i)} \geq \frac{1}{3} > \frac{5}{16} = \frac{\iota(F_4/P_2)}{\dim(F_4/P_2)-4} \qquad \text{ for all }i.
\] 

\item The rational homogeneous variety $F_4/P_3 \subset \mathbb P^{272}$ has dimension $20$ and Fano index $7$.
Since the semisimple part of $P_3$ is isomorphic to $\SL(3,\C) \times \SL(2,\C)$, by Proposition~\ref{dex of classical types},  
we know 
\[
\frac{\dex(\mathcal F_i)}{\rank(\mathcal F_i)} \geq \frac{1}{2} > \frac{7}{16} = \frac{\iota(F_4/P_3)}{\dim(F_4/P_3)-4} \qquad \text{ for all }i.
\] 
\end{enumerate}
Therefore, the result follows. 
\end{proof}

The classification in Table~\ref{table1} is a direct consequence of the following propositions. 

\begin{proposition} \label{G2/P1}
Let $\mathcal F$ be a completely reducible, globally generated, equivariant vector bundle over $G_2/P$, 
and $Z$ the zero locus of a general global section of $\mathcal F$.
\begin{enumerate}
\item 
If $Z \subset G_2/P_1$ is a fourfold with trivial canonical bundle, 
then $\mathcal F \cong \mathcal O(5)$. 
\item
If $Z \subset G_2/P_2$ is a fourfold with trivial canonical bundle, 
then $\mathcal F \cong \mathcal O(3)$. 
\end{enumerate}
\end{proposition}

\begin{proof}
Since the rational homogeneous variety $G_2/P_1 \subset \mathbb P(V_{G_2}(\varpi_1))=\mathbb P^{6}$ has dimension $5$ and Fano index~$5$, 
$\mathcal F$ is a line bundle $\mathcal O(5)$ of degree $5$. 
Similarly, the rational homogeneous variety $G_2/P_2 \subset \mathbb P(V_{G_2}(\varpi_2))=\mathbb P^{13}$ has dimension $5$ and Fano index~$3$. 
Hence, in this case $\mathcal F$ is a line bundle $\mathcal O(3)$ of degree 3. 
\end{proof}

\begin{remark}\label{rmk_Z_for_G2}
It is well-known that the rational homogeneous variety $G_2/P_1$ is isomorphic to a $5$-dimensional quadric $Q^5$ in $\mathbb{P}^6$ (e.g.,~\cite[p.~391]{FH}). 
Thus, the fourfold $Z \subset G_2/P_1$ in Proposition~\ref{G2/P1}(1) is the complete intersection of a smooth quadric hypersurface and a general quintic hypersurface in $\mathbb{P}^6$. 

Moreover, since the second fundamental representation $V_{G_2}(\varpi_2)$ is the adjoint representation $\mathfrak g_2$, 
$G_2/P_2$ is the adjoint variety of $G_2$. 
Thus, the fourfold $Z \subset G_2/P_2$ in Proposition~\ref{G2/P1}(2) is the complete intersection of the adjoint variety of $G_2$ and a general cubic hypersurface in $\mathbb{P}^{13}$. 
\end{remark}

The argument in Proposition~\ref{G2/P1} implies that 
a general global section of $\mathcal O(1) \oplus \mathcal O(4)$ or $\mathcal O(2) \oplus \mathcal O(3)$ over $G_2/P_1$ gives Calabi--Yau threefolds.  
As the equivariant bundle $\mathcal E_{\varpi_2}$ over $G_2/P_1$ has rank 2 and dex 3, 
we have $\dex(\mathcal E_{\varpi_2} \otimes \mathcal O(1)) = 3 \times 1 + 2 \times 1 = 5$; 
hence the zero locus of a general global section of $\mathcal E_{\varpi_2} \otimes \mathcal O(1)$ is a Calabi--Yau threefold. 
Similarly, a general global section of $\mathcal O(1) \oplus \mathcal O(2)$ 
or $\mathcal E_{\varpi_1} \otimes \mathcal O(1)$
over $G_2/P_2$ gives a Calabi--Yau threefold. 
See Table~\ref{table2}. 

\begin{proposition} \label{E6/P1}
Let $\mathcal F$ be a completely reducible, globally generated, equivariant vector bundle over the Cayley plane $E_6/P_1$. 
If $Z \subset E_6/P_1$ is a fourfold with trivial canonical bundle which is the zero locus of a general global section of $\mathcal F$, 
then $\mathcal F$ is isomorphic to $\mathcal O(1)^{\oplus 12}$. 
\end{proposition}

\begin{proof}
The (complex) Cayley plane $E_6/P_1$ has dimension 16 and Fano index 12. 
Recall that the semisimple part of the maximal parabolic subgroup $P_1 \subset E_6$ is isomorphic to $\Spin(10, \mathbb C)$. 
Since irreducible representations of $\Spin(10, \mathbb C)$ having dimension $\leq 16-4=12$ are only the trivial and standard representation, 
it suffices to consider the line bundles $\mathcal O(k)$ ($k \geq 1$) and the equivariant vector bundle $\mathcal E_{\varpi_6}$ of rank 10. 
However, since a general global section of $\mathcal E_{\varpi_6}$ vanishes nowhere as explained in~\cite[Section~1.2]{FM}, 
we obtain the result by Proposition~\ref{linear section}.  
\end{proof}

\begin{proposition} \label{E6/P2}
Let $\mathcal F$ be a completely reducible, globally generated, equivariant vector bundle over $E_6/P_2$.
If $Z \subset E_6/P_2$ is a fourfold with trivial canonical bundle which is the zero locus of a general global section of~$\mathcal F$, 
then $\mathcal F$ is isomorphic to 
either $\mathcal E_{\varpi_1}^{\oplus 2} \oplus \mathcal O(1)^{\oplus 5}$ or $\mathcal E_{\varpi_1} \oplus \mathcal E_{\varpi_6} \oplus \mathcal O(1)^{\oplus 5}$ or $\mathcal E_{\varpi_6}^{\oplus 2} \oplus \mathcal O(1)^{\oplus 5}$.
\end{proposition}

\begin{proof}
The rational homogeneous variety $E_6/P_2 \subset \mathbb P^{77}$ has dimension~$21$ and Fano index~$11$. 
Note that the semisimple part of the maximal parabolic subgroup $P_2 \subset E_6$ is isomorphic to $\SL(6, \mathbb C)$. 
Using the formula for the dex in Proposition~\ref{dex of classical types}, 
we obtain the dex of irreducible equivariant vector bundles with rank~$\leq 21-4=17$ as in~Table~\ref{table_dex_E6_P2}.
\begin{table}[H]
\begin{tabular}{c c c c c}
		\toprule
		$\lambda$ & $\varpi_1$  & $\varpi_3$  & $\varpi_5$  & $\varpi_6$
\\
		\midrule 
	$\rank (\mathcal E_{\lambda})$	&	$6$	&	$15$  &	$15$ 	&	$6$ 	
\\
	$\dex (\mathcal E_{\lambda})$	&	$3$	&	$15$  &	$15$ 	&	$3$ 	
\\
		\bottomrule
\end{tabular}
\caption{The irreducible equivariant vector bundles over $E_6/P_2$ with rank $\leq 17$.} 
\label{table_dex_E6_P2}
\end{table}
Considering the conditions in~\eqref{conditions}, 
$\mathcal F$ is isomorphic to 
either $\mathcal E_{\varpi_1}^{\oplus 2} \oplus \mathcal O(1)^{\oplus 5}$ or 
$\mathcal E_{\varpi_1} \oplus \mathcal E_{\varpi_6} \oplus \mathcal O(1)^{\oplus 5}$ or 
$\mathcal E_{\varpi_1}^{\oplus 2} \oplus \mathcal O(1)^{\oplus 5}$. 
\end{proof}

\begin{remark} 
Since the outer automorphism of $E_6$ corresponding to the symmetry of the Dynkin diagram induces a bundle isomorphism between $\mathcal E_{\varpi_1}$ and $\mathcal E_{\varpi_6}$, 
the zero locus $Z \subset E_6/P_2$ of a general global section of $\mathcal E_{\varpi_1}^{\oplus 2} \oplus \mathcal O(1)^{\oplus 5}$ in Proposition~\ref{E6/P2} 
is projectively equivalent to the zero locus of a general global section of~$\mathcal E_{\varpi_1} \oplus \mathcal E_{\varpi_6} \oplus \mathcal O(1)^{\oplus 5}$ or $\mathcal E_{\varpi_6}^{\oplus 2} \oplus \mathcal O(1)^{\oplus 5}$. 
\end{remark}

\begin{proposition} \label{E6/P3}
Let $\mathcal F$ be a completely reducible, globally generated, equivariant vector bundle over $E_6/P_3$.
If $Z \subset E_6/P_3$ is a fourfold with trivial canonical bundle which is the zero locus of a general global section of~$\mathcal F$, 
then $\mathcal F$ is isomorphic to 
either $\mathcal E_{\varpi_1}^{\oplus 3} \oplus \mathcal E_{\varpi_6}^{\oplus 3}$ or $\mathcal E_{\varpi_6}^{\oplus 4} \oplus \mathcal O(1)$.  
\end{proposition}

\begin{proof}
The rational homogeneous variety $E_6/P_3$ has dimension $25$ and Fano index $9$. 
Note that the rational homogeneous variety $E_6/P_5 \subset \mathbb P^{350}$ is projectively equivalent to $E_6/P_3 \subset \mathbb P^{350}$
because the highest weight $E_6$-module $V_{E_6}(\varpi_5)$ is dual to $V_{E_6}(\varpi_3)$. 

Since the semisimple part of the maximal parabolic subgroup $P_3 \subset E_6$ is isomorphic to $\SL(2, \mathbb C) \times \SL(5, \mathbb C)$, 
by Proposition~\ref{dex of classical types} irreducible equivariant vector bundles with rank $\leq 25-4=21$ and dex $\leq 9$ are given in Table~\ref{table_dex_E6_P3}.
\begin{table}[H]
\begin{tabular}{c c c c c c c}
		\toprule
		$\lambda$ & $\varpi_1$   & $2\varpi_1$ & $3\varpi_1$ & $\varpi_2$ & $\varpi_5$ & $\varpi_6$
\\
		\midrule 
	$\rank (\mathcal E_{\lambda})$	&	$2$	&	$3$  &	$4$ 	&	$5$ 	&	$10$  &	$5$ 
\\
	$\dex (\mathcal E_{\lambda})$	&	$1$	&	$3$  &	$6$ 	&	$3$ 	&	$8$  &	$2$ 
\\
		\bottomrule
\end{tabular}
\caption{The irreducible equivariant bundles over $E_6/P_3$ with rank $\leq 21$ and dex $\leq 9$.} 
\label{table_dex_E6_P3}
\end{table}
Except for $\lambda=\varpi_6$, we have  
\[
\frac{\dex(\mathcal E_{\lambda})}{\rank(\mathcal E_{\lambda})} > \frac{9}{21} = \frac{\iota(E_6/P_3)}{\dim(E_6/P_3)-4}
\] 
so that $\mathcal F$ must have $\mathcal E_{\varpi_6}$ as direct summands.  
Repeating this process,  
$\mathcal F$ is isomorphic to 
either $\mathcal E_{\varpi_1}^{\oplus 3} \oplus \mathcal E_{\varpi_6}^{\oplus 3}$ or $\mathcal E_{\varpi_6}^{\oplus 4} \oplus \mathcal O(1)$  
from the conditions in~\eqref{conditions}.  
\end{proof}

The argument in Proposition \ref{E6/P3} implies that 
a general global section of $\mathcal E_{\varpi_1} \oplus \mathcal E_{\varpi_6}^{\oplus 4}$ over $E_6/P_3$ gives a Calabi--Yau threefold. 
See Table~\ref{table2}. 

\begin{proposition} \label{E7/P1}
Let $\mathcal F$ be a completely reducible, globally generated, equivariant vector bundle over $E_7/P_1$.
If $Z \subset E_7/P_1$ is a fourfold with trivial canonical bundle which is the zero locus of a general global section of $\mathcal F$, 
then $\mathcal F$ is isomorphic to $\mathcal E_{\varpi_7}^{\oplus 2} \oplus \mathcal O(1)^{\oplus 5}$.  
\end{proposition}

\begin{proof}
As the fundamental $E_7$-module $V_{E_7}(\varpi_1)$ is the adjoint representation $\mathfrak e_7$,  
the rational homogeneous variety $E_7/P_1$ is the adjoint variety of $E_7$. 
The $E_7$-adjoint variety $E_7/P_1 \subset \mathbb P^{132}$ has dimension $33$ and Fano index $17$. 
Since the semisimple part of the maximal parabolic subgroup $P_1 \subset E_7$ is isomorphic to $\Spin(12, \mathbb C)$, 
and irreducible representations of $\Spin(12, \mathbb C)$ having dimension $\leq 33-4=29$ are only the trivial and standard representation, 
it suffices to consider the line bundles $\mathcal O(k)$ ($k \geq 1$) and the equivariant vector bundle $\mathcal E_{\varpi_7}$ of rank $12$. 

From the direct computation using the action of the Weyl group, 
we know that the sum of all weights of the irreducible $P_1$-module with highest weight $\varpi_7$ is equal to $6 \varpi_1$. 
Therefore, $\dex(\varpi_7)=6$ and 
$\mathcal F$ is isomorphic to $\mathcal E_{\varpi_7}^{\oplus 2} \oplus \mathcal O(1)^{\oplus 5}$ from the conditions in~\eqref{conditions}.  
\end{proof}

\begin{proposition} \label{F4/P1}
Let $\mathcal F$ be a completely reducible, globally generated, equivariant vector bundle over $F_4/P_1$.
If $Z \subset F_4/P_1$ is a fourfold with trivial canonical bundle which is the zero locus of a general global section of~$\mathcal F$, 
then $\mathcal F$ is isomorphic to $\mathcal E_{\varpi_4} \oplus \mathcal O(1)^{\oplus 5}$.
\end{proposition}

\begin{proof}
The rational homogeneous variety $F_4/P_1 \subset \mathbb P^{51}$ has dimension $15$ and Fano index $8$. 
Since the semisimple part of the maximal parabolic subgroup $P_1 \subset F_4$ is isomorphic to $\Sp(6, \mathbb C)$, 
and irreducible representations of $\Sp(6, \mathbb C)$ having dimension $\leq 15-4=11$ are only the trivial and standard representation, 
it suffices to consider the line bundles $\mathcal O(k)$ ($k \geq 1$) and the equivariant vector bundle $\mathcal E_{\varpi_4}$ of rank $6$. 

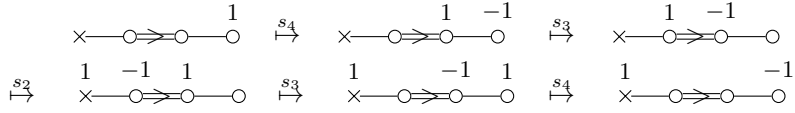
\begin{center}
\begin{tabular}{rrl}
\begin{tikzpicture}[scale=.5, node distance=0.5cm, baseline=-0.5ex]

\node[point] (1) {};
\node[state] (2) [right = of 1] {};
\node[state] (3) [right = of 2] {};
\node[state] (4) [right =of 3] {};

\draw (1)--(2)
(3)--(4);
\draw[double line] (2)-- node{$>$} (3);

\foreach \c [count = \x from 1] in {{},{},{},{1}} 
\node[above = 0.1cm] at (\x) {\small $\c$};

\end{tikzpicture} 
& $\stackrel{s_4}{\mapsto}\quad$
\begin{tikzpicture}[scale=.5, node distance=0.5cm, baseline=-0.5ex]

\node[point] (1) {};
\node[state] (2) [right = of 1] {};
\node[state] (3) [right = of 2] {};
\node[state] (4) [right =of 3] {};

\draw (1)--(2)
(3)--(4);
\draw[double line] (2)-- node{$>$} (3);

\foreach \c [count = \x from 1] in {{},{},{1},{-1}} 
\node[above = 0.1cm] at (\x) {\small $\c$};

\end{tikzpicture} 
& $\stackrel{s_3}{\mapsto}\quad$
\begin{tikzpicture}[scale=.5, node distance=0.5cm, baseline=-0.5ex]

\node[point] (1) {};
\node[state] (2) [right = of 1] {};
\node[state] (3) [right = of 2] {};
\node[state] (4) [right =of 3] {};

\draw (1)--(2)
(3)--(4);
\draw[double line] (2)-- node{$>$} (3);

\foreach \c [count = \x from 1] in {{},{1},{-1},{}} 
\node[above = 0.1cm] at (\x) {\small $\c$};
\end{tikzpicture} \\
$\stackrel{s_2}{\mapsto}\quad$
\begin{tikzpicture}[scale=.5, node distance=0.5cm, baseline=-0.5ex]

\node[point] (1) {};
\node[state] (2) [right = of 1] {};
\node[state] (3) [right = of 2] {};
\node[state] (4) [right =of 3] {};

\draw (1)--(2)
(3)--(4);
\draw[double line] (2)-- node{$>$} (3);

\foreach \c [count = \x from 1] in {{1},{-1},{1},{}} 
\node[above = 0.1cm] at (\x) {\small $\c$};

\end{tikzpicture}
&$\stackrel{s_3}{\mapsto}\quad$
\begin{tikzpicture}[scale=.5, node distance=0.5cm, baseline=-0.5ex]

\node[point] (1) {};
\node[state] (2) [right = of 1] {};
\node[state] (3) [right = of 2] {};
\node[state] (4) [right =of 3] {};

\draw (1)--(2)
(3)--(4);
\draw[double line] (2)-- node{$>$} (3);

\foreach \c [count = \x from 1] in {{1},{},{-1},{1}} 
\node[above = 0.1cm] at (\x) {\small $\c$};

\end{tikzpicture}
&$\stackrel{s_4}{\mapsto}\quad $
\begin{tikzpicture}[scale=.5, node distance=0.5cm, baseline=-0.5ex]

\node[point] (1) {};
\node[state] (2) [right = of 1] {};
\node[state] (3) [right = of 2] {};
\node[state] (4) [right =of 3] {};

\draw (1)--(2)
(3)--(4);
\draw[double line] (2)-- node{$>$} (3);

\foreach \c [count = \x from 1] in {{1},{},{},{-1}} 
\node[above = 0.1cm] at (\x) {\small $\c$};

\end{tikzpicture}
\end{tabular}
\end{center}

From the direct computation using the action of the Weyl group, 
we know that the sum of all weights of the irreducible $P_1$-module with highest weight $\varpi_4$ is equal to $3 \varpi_1$. 
Therefore, $\dex(\varpi_4)=3$ and 
$\mathcal F$ is isomorphic to $\mathcal E_{\varpi_4} \oplus \mathcal O(1)^{\oplus 5}$ from the conditions in~\eqref{conditions}.  
\end{proof}

\begin{proposition} \label{F4/P4}
Let $\mathcal F$ be a completely reducible, globally generated, equivariant vector bundle over $F_4/P_4$.
If $Z \subset F_4/P_4$ is a fourfold with trivial canonical bundle which is the zero locus of a general global section of~$\mathcal F$, 
then $\mathcal F$ is isomorphic to 
either $\mathcal O(1)^{\oplus 11}$ or $\mathcal E_{\varpi_1} \oplus \mathcal O(1)^{\oplus 4}$. 
\end{proposition}

\begin{proof}
The rational homogeneous variety $F_4/P_4 \subset \mathbb P^{25}$ has dimension $15$ and Fano index $11$. 
Since the semisimple part of the maximal parabolic subgroup $P_4 \subset F_4$ is isomorphic to $\Spin(7, \mathbb C)$, 
and irreducible representations of $\Spin(7, \mathbb C)$ having dimension $\leq 15-4=11$ are only the trivial, standard and spin representation, 
it suffices to consider the line bundles $\mathcal O(k)$ ($k \geq 1$), the equivariant vector bundles $\mathcal E_{\varpi_1}$ of rank $7$ and $\mathcal E_{\varpi_3}$ of rank $8$. 

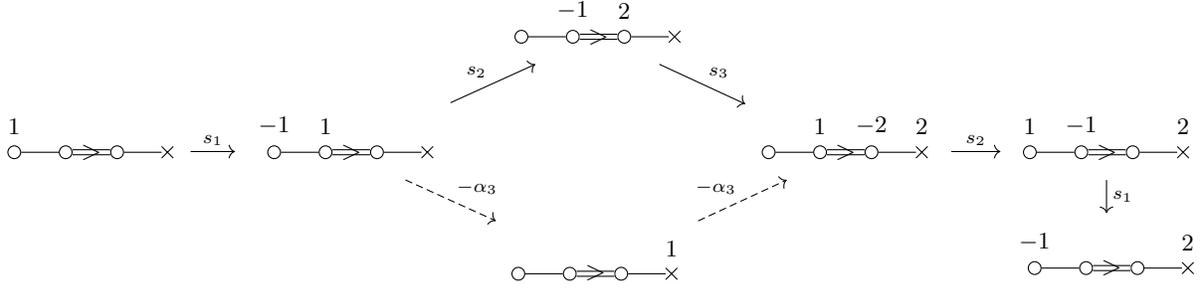
\[
\begin{tikzcd}[column sep = 4ex, row sep = 3ex]
&&
\begin{tikzpicture}[anchor = center, scale=.5, node distance=0.5cm, baseline=-0.5ex]			
	\node[state] (1) {};
	\node[state] (2) [right = of 1] {};
	\node[state] (3) [right = of 2] {};
	\node[point] (4) [right =of 3] {};
	\draw (1)--(2)
		(3)--(4);
	\draw[double line] (2)-- node{$>$} (3);
			
	\foreach \c [count = \x from 1] in {{},{-1},{2},{}} 
	\node[above = 0.1cm] at (\x) {\small $\c$};
			
\end{tikzpicture} 
\arrow[rd, "s_3"]
\\
\begin{tikzpicture}[anchor = center, scale=.5, node distance=0.5cm, baseline = -0.5ex]	
	\node[state] (1) {};
	\node[state] (2) [right = of 1] {};
	\node[state] (3) [right = of 2] {};
	\node[point] (4) [right =of 3] {};

	\draw (1)--(2)
		(3)--(4);
	\draw[double line] (2)-- node{$>$} (3);
				
	\foreach \c [count = \x from 1] in {{1},{},{},{}} 
	\node[above = 0.1cm] at (\x) {\small $\c$};

\end{tikzpicture} 
\arrow[r, "s_1"]
&
\begin{tikzpicture}[anchor = center, scale=.5, node distance=0.5cm, baseline = -0.5ex]
	\node[state] (1) {};
	\node[state] (2) [right = of 1] {};
	\node[state] (3) [right = of 2] {};
	\node[point] (4) [right =of 3] {};
					
	\draw (1)--(2)
		(3)--(4);			
	\draw[double line] (2)-- node{$>$} (3);
			
	\foreach \c [count = \x from 1] in {{-1},{1},{},{}} 
	\node[above = 0.1cm] at (\x) {\small $\c$};
\end{tikzpicture} 
\arrow[ur, "s_2"] 
\arrow[dr, dashed, "-\alpha_3"]
&& \begin{tikzpicture}[anchor = center, scale=.5, node distance=0.5cm, baseline=-0.5ex]
			
	\node[state] (1) {};
	\node[state] (2) [right = of 1] {};
	\node[state] (3) [right = of 2] {};
	\node[point] (4) [right =of 3] {};
	\draw (1)--(2)
		(3)--(4);
	\draw[double line] (2)-- node{$>$} (3);
			
	\foreach \c [count = \x from 1] in {{},{1},{-2},{2}} 
	\node[above = 0.1cm] at (\x) {\small $\c$};
			
\end{tikzpicture} 
\arrow[r, "s_2"]
& \begin{tikzpicture}[anchor = center, scale=.5, node distance=0.5cm, baseline=-0.5ex]
	\node[state] (1) {};
	\node[state] (2) [right = of 1] {};
	\node[state] (3) [right = of 2] {};
	\node[point] (4) [right =of 3] {};
	\draw (1)--(2)
		(3)--(4);
	\draw[double line] (2)-- node{$>$} (3);
			
	\foreach \c [count = \x from 1] in {{1},{-1},{},{2}} 
	\node[above = 0.1cm] at (\x) {\small $\c$};
\end{tikzpicture} \arrow[d, "s_1"] \\
&& \begin{tikzpicture}[anchor = mid, scale=.5, node distance=0.5cm, baseline=-0.5ex]
			
	\node[state] (1) {};
	\node[state] (2) [right = of 1] {};
	\node[state] (3) [right = of 2] {};
	\node[point] (4) [right =of 3] {};
	\draw (1)--(2)
		(3)--(4);
	\draw[double line] (2)-- node{$>$} (3);
	\foreach \c [count = \x from 1] in {{},{},{},{1}} 
	\node[above = 0.1cm] at (\x) {\small $\c$};
			
\end{tikzpicture}
\arrow[ur, dashed, "-\alpha_3"]
&& \begin{tikzpicture}[anchor = center, scale=.5, node distance=0.5cm, baseline=-0.5ex]
	\node[state] (1) {};
	\node[state] (2) [right = of 1] {};
	\node[state] (3) [right = of 2] {};
	\node[point] (4) [right =of 3] {};
	\draw (1)--(2)
		(3)--(4);
	\draw[double line] (2)-- node{$>$} (3);
	\foreach \c [count = \x from 1] in {{-1},{},{},{2}} 
	\node[above = 0.1cm] at (\x) {\small $\c$};
\end{tikzpicture} 
\end{tikzcd}
\]
From the direct computation, 
the sum of all weights of the irreducible $P_4$-module with highest weight $\varpi_1$ is equal to $7 \varpi_1$. 
Therefore, $\dex(\varpi_1)=7$ and 
$\mathcal F$ is isomorphic to $\mathcal E_{\varpi_1} \oplus \mathcal O(1)^{\oplus 4}$ from the conditions in~\eqref{conditions}.  

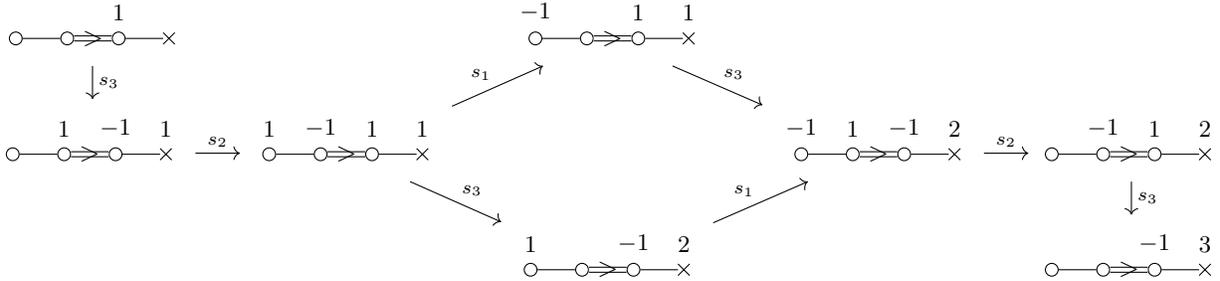
\[
\begin{tikzcd}[column sep = 4ex, row sep = 3ex]
\begin{tikzpicture}[anchor = center, scale=.5, node distance=0.5cm, baseline=-0.5ex]
	\node[state] (1) {};
	\node[state] (2) [right = of 1] {};
	\node[state] (3) [right = of 2] {};
	\node[point] (4) [right =of 3] {};

	\draw (1)--(2)
	(3)--(4);
	\draw[double line] (2)-- node{$>$} (3);
	
	\foreach \c [count = \x from 1] in {{},{},{1},{}} 
	\node[above = 0.1cm] at (\x) {\small $\c$};
	
\end{tikzpicture} 
\arrow[d, "s_3"]
&&	\begin{tikzpicture}[anchor = center, scale=.5, node distance=0.5cm, baseline=-0.5ex]

\node[state] (1) {};
\node[state] (2) [right = of 1] {};
\node[state] (3) [right = of 2] {};
\node[point] (4) [right =of 3] {};

\draw (1)--(2)
(3)--(4);
\draw[double line] (2)-- node{$>$} (3);

\foreach \c [count = \x from 1] in {{-1},{},{1},{1}} 
\node[above = 0.1cm] at (\x) {\small $\c$};

\end{tikzpicture}  
\arrow[rd, "s_3"]
\\
\begin{tikzpicture}[anchor = center, scale=.5, node distance=0.5cm, baseline=-0.5ex]

\node[state] (1) {};
\node[state] (2) [right = of 1] {};
\node[state] (3) [right = of 2] {};
\node[point] (4) [right =of 3] {};

\draw (1)--(2)
(3)--(4);
\draw[double line] (2)-- node{$>$} (3);

\foreach \c [count = \x from 1] in {{},{1},{-1},{1}} 
\node[above = 0.1cm] at (\x) {\small $\c$};

\end{tikzpicture} 
\arrow[r, "s_2"]
& \begin{tikzpicture}[anchor = center, scale=.5, node distance=0.5cm, baseline=-0.5ex]
	
	\node[state] (1) {};
	\node[state] (2) [right = of 1] {};
	\node[state] (3) [right = of 2] {};
	\node[point] (4) [right =of 3] {};

	\draw (1)--(2)
	(3)--(4);
	\draw[double line] (2)-- node{$>$} (3);
	
	\foreach \c [count = \x from 1] in {{1},{-1},{1},{1}} 
	\node[above = 0.1cm] at (\x) {\small $\c$};
	
\end{tikzpicture} 
\arrow[ru, "s_1"]
\arrow[rd, "s_3"]
&& \begin{tikzpicture}[anchor = center, scale=.5, node distance=0.5cm, baseline = -0.5ex]

\node[state] (1) {};
\node[state] (2) [right = of 1] {};
\node[state] (3) [right = of 2] {};
\node[point] (4) [right =of 3] {};

\draw (1)--(2)
(3)--(4);
\draw[double line] (2)-- node{$>$} (3);

\foreach \c [count = \x from 1] in {{-1},{1},{-1},{2}} 
\node[above = 0.1cm] at (\x) {\small $\c$};
\end{tikzpicture} 
\arrow[r, "s_2"]
& \begin{tikzpicture}[anchor = center, scale=.5, node distance=0.5cm, baseline = -0.5ex]
	
	\node[state] (1) {};
	\node[state] (2) [right = of 1] {};
	\node[state] (3) [right = of 2] {};
	\node[point] (4) [right =of 3] {};

	\draw (1)--(2)
	(3)--(4);
	\draw[double line] (2)-- node{$>$} (3);
	
	\foreach \c [count = \x from 1] in {{},{-1},{1},{2}} 
	\node[above = 0.1cm] at (\x.center) {\small $\c$};
	
\end{tikzpicture}  
\arrow[d, "s_3"]
\\
&& 	\begin{tikzpicture}[anchor = center, scale=.5, node distance=0.5cm, baseline = -0.5ex]

\node[state] (1) {};
\node[state] (2) [right = of 1] {};
\node[state] (3) [right = of 2] {};
\node[point] (4) [right =of 3] {};

\draw (1)--(2)
(3)--(4);
\draw[double line] (2)-- node{$>$} (3);

\foreach \c [count = \x from 1] in {{1},{},{-1},{2}} 
\node[above = 0.1cm] at (\x) {\small $\c$};

\end{tikzpicture} 
\arrow[ru, "s_1"]
&& 	\begin{tikzpicture}[anchor = center, scale=.5, node distance=0.5cm, baseline = -0.5ex]
	
	\node[state] (1) {};
	\node[state] (2) [right = of 1] {};
	\node[state] (3) [right = of 2] {};
	\node[point] (4) [right =of 3] {};

	\draw (1)--(2)
	(3)--(4);
	\draw[double line] (2)-- node{$>$} (3);
	
	\foreach \c [count = \x from 1] in {{},{},{-1},{3}} 
	\node[above = 0.1cm] at (\x) {\small $\c$};
	
\end{tikzpicture} 
\end{tikzcd}
\]

On the other hand, since the sum of all weights of the irreducible $P_4$-module with highest weight $\varpi_3$ is equal to $12 \varpi_1$, 
we get $\dex(\varpi_3)=12$. 
Thus $\mathcal F$ can not have $\mathcal E_{\varpi_3}$ as a direct summand. 
\end{proof}

\begin{remark}
Since the 27-dimensional fundamental $E_6$-module $V_{E_6}(\varpi_1)$ may be regarded as an $F_4$-module using the embedding of $F_4$ in $E_6$
and it decomposes as $V_{F_4}(\varpi_4) \oplus V_{F_4}(0)$ (see~\cite[Proposition~13.32]{Carter}), 
the rational homogeneous variety $F_4/P_4$ is a general hyperplane section of the Cayley plane $E_6/P_1 \subset \mathbb P^{26}$ (see~\cite[Section~6.3]{LM}).
Thus, the zero locus $Z \subset F_4/P_4$ of a general global section of~$\mathcal O(1)^{\oplus 11}$ in Proposition \ref{F4/P4} is the same as the zero locus of a general global section of~$\mathcal O(1)^{\oplus 12}$ over $E_6/P_1$ in Proposition~\ref{E6/P1}. 
\end{remark}

\section{Computations of Hodge numbers}
\label{section_Hodge_numbers}

The Borel--Weil--Bott theorem (Theorem~\ref{BBW}) is a powerful tool for computing cohomologies of equivariant vector bundles over homogeneous varieties. 
Using  the Koszul complex, the conormal sequence and the exact sequences derived from this, 
the Hodge numbers of Calabi--Yau fourfolds and threefolds in Table~\ref{table1} and Table~\ref{table2} can be calculated by the similar way as in~\cite{Kuchle}. 

All examples have $h^{0, 0}(Z)=h^0(\mathcal O_Z)=1$ and $h^{0,1}(Z)=h^1(\mathcal O_Z)=0$ which mean that they are irreducible and Calabi--Yau in the strict sense. 

\subsection{Computations of $h^{0, q}(Z)$}

Let $G/P$ be a rational homogeneous variety of Picard number one, and $\mathcal F$ be a completely reducible, globally generated, equivariant vector bundle over $G/P$ of rank $n-4$, where $n=\dim(G/P)$. 
Using a similar argument as in Example~\ref{Beauville--Donagi},
the Borel--Weil--Bott theorem and the Koszul complex
$$0\to \Exterior^{n-4} \mathcal F^* \to \Exterior^{n-5} \mathcal F^* \to \cdots \to \Exterior^2 \mathcal F^* \to \mathcal F^* \to \mathcal O_{G/P} \to \mathcal O_Z \to 0$$
allow us to calculate the Hodge numbers $h^{0, q}(Z) = \dim H^q(Z, \mathcal O_Z)$ for the zero locus $Z$ of a general global section of $\mathcal F$. 

\begin{proposition}\label{prop_h0q_Z}
Let $G$ be a simple Lie group of exceptional type.
Let $G/P$ be a rational homogeneous variety of Picard number one, and $\mathcal F$ be a completely reducible, globally generated, equivariant vector bundle over $G/P$ of rank $n-4$. 
If the zero locus $Z$ of a general global section of $\mathcal F$ is a fourfold with trivial canonical bundle, 
then $h^{0, 0}(Z) = h^{0, 4}(Z)=1$ and $h^{0, q}(Z)=0$ for $q = 1,2,3$. 
\end{proposition}

\begin{proof}
First, if $Z$ is of No. 1, No. 8, No. 9 in the Table \ref{table1}, then $h^{0, 2}(Z)=0$ 
by Proposition \ref{direct sum of line bundles} and $h^{0, 4}(Z) = h^{0, 0}(Z)=1$ from the straightforward computations. 
Using the same argument as Example~\ref{Beauville--Donagi}, 
we compute the Hodge numbers $h^{0, 0}(Z), h^{0,2}(Z), h^{0, 4}(Z)$ of the other Calabi--Yau fourfolds in Table~\ref{table1}.

No. 2. $\mathcal F = \mathcal E_{\varpi_1}^{\oplus 2} \oplus \mathcal O(1)^{\oplus 5}$ over $E_6/P_2$

No. 3. $\mathcal F = \mathcal E_{\varpi_1}^{\oplus 3} \oplus \mathcal E_{\varpi_6}^{\oplus 3}$ over $E_6/P_3$

No. 4. $\mathcal F = \mathcal E_{\varpi_6}^{\oplus 4} \oplus \mathcal O(1)$ over $E_6/P_3$

No. 5. $\mathcal F = \mathcal E_{\varpi_7}^{\oplus 2} \oplus \mathcal O(1)^{\oplus 5}$ over $E_7/P_1$

No. 6. $\mathcal F = \mathcal E_{\varpi_4} \oplus \mathcal O(1)^{\oplus 5}$ over $F_4/P_1$

No. 7. $\mathcal F = \mathcal E_{\varpi_1} \oplus \mathcal O(1)^{\oplus 4}$ over $F_4/P_4$

For instance,  
let $Z$ be the zero locus of a general global section of the equivariant vector bundle $\mathcal F = \mathcal E_{\varpi_1} \oplus \mathcal O(1)^{\oplus 4}$ over $F_4/P_4$. 
As we have already seen in Proposition \ref{F4/P4}, the reductive part of the maximal parabolic subgroup $P_4 \subset F_4$ is isomorphic to $\Spin(7, \mathbb C) \times \mathbb C^*$. 
But weights of $P_4$-modules have to be computed inside the weight lattice of $F_4$. 
Using representations of $\Spin(10, \mathbb C)$ and adjusting the first Chern class, the wedge powers of the irreducible equivariant bundle $\mathcal E_{\varpi_1}$ are described as follows.
\begin{equation}\label{eq_exterior_Epi1_E6}
\begin{array}{lll}
\Exterior^2 \mathcal E_{\varpi_1}=\mathcal E_{\varpi_2}, 
& \Exterior^3 \mathcal E_{\varpi_1} =\mathcal E_{2\varpi_3},  
& \Exterior^4 \mathcal E_{\varpi_1} =\mathcal E_{2\varpi_3 + \varpi_4}=\mathcal E_{2\varpi_3}(1), \\ 
\Exterior^5 \mathcal E_{\varpi_1} =\mathcal E_{\varpi_2 + 3\varpi_4}=\mathcal E_{\varpi_2}(3), 
& \Exterior^6 \mathcal E_{\varpi_1}=\mathcal E_{\varpi_1 + 5\varpi_4}=\mathcal E_{\varpi_1}(5), 
& \Exterior^7 \mathcal E_{\varpi_1}=\mathcal E_{7\varpi_4}=\mathcal O(7). 
\end{array}
\end{equation}

We have the Koszul complex associated to a general global section of the vector bundle $\mathcal F$: 
$$0\to \Exterior^{11} \mathcal F^* \to \Exterior^{10} \mathcal F^* \to \cdots \to \Exterior^2 \mathcal F^* \to \mathcal F^* \to \mathcal O_{F_4/P_4} \to \mathcal O_Z \to 0.$$
Using~\eqref{eq_exterior_Epi1_E6} and $\mathcal E_{\varpi_1}^* = \mathcal E_{\varpi_1}(-2)$, we know  
\begin{align*}
\Exterior^2 \mathcal F^* 
&= \Exterior^2 \mathcal E_{\varpi_1}^* \oplus \mathcal E_{\varpi_1}^*(-1)^{\oplus 4} \oplus \mathcal O(-2)^{\oplus 6}  
= \mathcal E_{\varpi_2}(-4) \oplus \mathcal E_{\varpi_1}(-3)^{\oplus 4} \oplus \mathcal O(-2)^{\oplus 6}, \\
\Exterior^3 \mathcal F^* 
&= \Exterior^3 \mathcal E_{\varpi_1}^* \oplus \Exterior^2 \mathcal E_{\varpi_1}^*(-1)^{\oplus 4} \oplus \mathcal E_{\varpi_1}^*(-2)^{\oplus 6} \oplus \mathcal O(-3)^{\oplus 4} \\
&= \mathcal E_{2\varpi_3}(-6) \oplus \mathcal E_{\varpi_2}(-5)^{\oplus 4} \oplus \mathcal E_{\varpi_1}(-4)^{\oplus 6} \oplus \mathcal O(-3)^{\oplus 4}, \\
\Exterior^4 \mathcal F^* 
&= \mathcal E_{2\varpi_3}(-7) \oplus \mathcal E_{2\varpi_3}(-7)^{\oplus 4} \oplus \mathcal E_{\varpi_2}(-6)^{\oplus 6} \oplus \mathcal E_{\varpi_1}(-5)^{\oplus 4} \oplus \mathcal O(-4), \\
\Exterior^5 \mathcal F^* 
&= \mathcal E_{\varpi_2}(-7) \oplus \mathcal E_{2\varpi_3}(-8)^{\oplus 4} \oplus \mathcal E_{2\varpi_3}(-8)^{\oplus 6} \oplus \mathcal E_{\varpi_2}(-7)^{\oplus 4} \oplus \mathcal E_{\varpi_1}(-6), \\
\Exterior^6 \mathcal F^* 
&= \mathcal E_{\varpi_1}(-7) \oplus \mathcal E_{\varpi_2}(-8)^{\oplus 4} \oplus \mathcal E_{2\varpi_3}(-9)^{\oplus 6} \oplus \mathcal E_{2\varpi_3}(-9)^{\oplus 4} \oplus \mathcal E_{\varpi_2}(-8), \\
\Exterior^7 \mathcal F^* 
&= \mathcal O(-7) \oplus \mathcal E_{\varpi_1}(-8)^{\oplus 4} \oplus \mathcal E_{\varpi_2}(-9)^{\oplus 6} \oplus \mathcal E_{2\varpi_3}(-10)^{\oplus 4} \oplus \mathcal E_{2\varpi_3}(-10), \\
\Exterior^8 \mathcal F^* 
&= \mathcal O(-8)^{\oplus 4} \oplus \mathcal E_{\varpi_1}(-9)^{\oplus 6} \oplus \mathcal E_{\varpi_2}(-10)^{\oplus 4} \oplus \mathcal E_{2\varpi_3}(-11), \\
\Exterior^9 \mathcal F^* 
&= \mathcal O(-9)^{\oplus 6} \oplus \mathcal E_{\varpi_1}(-10)^{\oplus 4} \oplus \mathcal E_{\varpi_2}(-11), \\
\Exterior^{10} \mathcal F^* &= \mathcal O(-10)^{\oplus 4} \oplus \mathcal E_{\varpi_1}(-11), \\ 
\Exterior^{11} \mathcal F^* &= \mathcal O(-11). 
\end{align*}
Then each irreducible equivariant bundle in $\Exterior^p \mathcal F^*$ is acyclic for $1 \leq p \leq 10$.  
For example, we can check that $\mathcal E_{2\varpi_3}(-7)$ is acyclic by the Borel--Weil--Bott theorem 
because the weights $2\varpi_3 -7 \varpi_4 + \rho = \varpi_1 + \varpi_2 + 3\varpi_3 -6 \varpi_4$ is singular. 
Indeed, $(s_3 s_1 s_2 s_3 s_4) (\varpi_1 + \varpi_2 + 3\varpi_3 -6 \varpi_4)$ is orthogonal to the simple root $\alpha_2$. 
On the other hand, by the Serre duality, we get 
\[
H^{15}(F_4/P_4, \Exterior^{11} \mathcal F^* ) = H^{15}(F_4/P_4, \mathcal O(-11) ) \cong H^{0}(F_4/P_4, \mathcal O)^* \cong \mathbb C,
\] 
which implies $H^4(Z, \mathcal O_Z)\cong\mathbb C$. 
Consequently, we conclude $h^0(Z, \mathcal O_Z)=h^4(Z, \mathcal O_Z)=1$ and $h^q(Z, \mathcal O_Z)=0$ for $q=1, 2, 3$. 

We consider one more example. Let $Z$ be the zero locus of a general global section of the equivariant vector bundle $\mathcal F = \mathcal E_{\varpi_6}^{\oplus 4} \oplus \mathcal O(1)$ over $E_6/P_3$. 
Using the computer program \texttt{SageMath}~\cite{SAGE}, we provide weights appearing in each wedge power of the bundle $\mathcal F^{\ast}$ in Table~\ref{table_E6P3_all_weights}.
Each vector $(\lambda_1,\dots,\lambda_6)$ represents the weight $\lambda_1 \varpi_1 + \cdots +\lambda_6 \varpi_6$.
For example, the third wedge power of $\mathcal F^{\ast}$ is 
\[
\Exterior^3 \mathcal F^{\ast} = 
\mathcal E_{\varpi_4}(-3)^{\oplus 10} \oplus 
\mathcal E_{\varpi_5}(-2)^{\oplus 20} \oplus 
\mathcal E_{\varpi_2 + \varpi_4}(-3)^{\oplus 20} \oplus 
\mathcal E_{2\varpi_2}(-3)^{\oplus 6} \oplus
\mathcal E_{3 \varpi_2}(-3)^{\oplus 4}
\]
and this decomposition corresponds to vectors $(0, 0, -3, 1, 0, 0)$, 
$(0, 0, -2, 0, 1, 0)$, 
$(0, 1, -3, 1, 0, 0)$, 
$(0, 2, -3, 0, 0, 0)$, 
$(0, 3, -3, 0, 0, 0)$. 
\begin{table}[b]
\begin{tabularx}{\textwidth}{l|X}
\toprule
$p$ & weights appearing in $\Exterior^p \mathcal F^{\ast}$ \\
\midrule 
$1$ & 
\footnotesize{$(0, 0, -1, 0, 0, 0), 
(0, 1, -1, 0, 0, 0)$}\\
$2$ & 
\footnotesize{$(0, 0, -2, 1, 0, 0), 
(0, 1, -2, 0, 0, 0), 
(0, 2, -2, 0, 0, 0)$} \\
$3$&
\footnotesize{$(0, 0, -3, 1, 0, 0), 
(0, 0, -2, 0, 1, 0), 
(0, 1, -3, 1, 0, 0), 
(0, 2, -3, 0, 0, 0), 
(0, 3, -3, 0, 0, 0)$}\\
$4$&
\footnotesize{$(0, 0, -4, 2, 0, 0), 
(0, 0, -3, 0, 1, 0), 
(0, 0, -2, 0, 0, 1), 
(0, 1, -4, 1, 0, 0), 
(0, 1, -3, 0, 1, 0), 
(0, 2, -4, 1, 0, 0),(0, 3, -4, 0, 0, 0)$,} 
\footnotesize{$(0, 4, -4, 0, 0, 0)$} \\
$5$ & 
\footnotesize{$(0, 0, -5, 2, 0, 0), 
(0, 0, -4, 1, 1, 0), 
(0, 0, -3, 0, 0, 1), 
(0, 0, -2, 0, 0, 0), 
(0, 1, -5, 2, 0, 0), 
(0, 1, -4, 0, 1, 0), 
(0, 1, -3, 0, 0, 1),$} 
\footnotesize{$(0, 2, -5, 1, 0, 0), 
(0, 2, -4, 0, 1, 0), 
(0, 3, -5, 1, 0, 0), 
(0, 4, -5, 0, 0, 0)$} \\
$6$ &
\footnotesize{$(0, 0, -6, 3, 0, 0), 
(0, 0, -5, 1, 1, 0), 
(0, 0, -4, 0, 2, 0), 
(0, 0, -4, 1, 0, 1), 
(0, 0, -3, 0, 0, 0), 
(0, 1, -6, 2, 0, 0), 
(0, 1, -5, 1, 1, 0), $}
\footnotesize{$(0, 1, -4, 0, 0, 1), 
(0, 1, -3, 0, 0, 0), 
(0, 2, -6, 2, 0, 0), 
(0, 2, -5, 0, 1, 0), 
(0, 2, -4, 0, 0, 1), 
(0, 3, -6, 1, 0, 0), 
(0, 3, -5, 0, 1, 0)$} \\
$7$ &
\footnotesize{$(0, 0, -7, 3, 0, 0), 
(0, 0, -6, 2, 1, 0), 
(0, 0, -5, 0, 2, 0), 
(0, 0, -5, 1, 0, 1), 
(0, 0, -4, 0, 1, 1), 
(0, 0, -4, 1, 0, 0), 
(0, 1, -7, 3, 0, 0), $}
\footnotesize{$(0, 1, -6, 1, 1, 0), 
(0, 1, -5, 0, 2, 0), 
(0, 1, -5, 1, 0, 1), 
(0, 1, -4, 0, 0, 0), 
(0, 2, -7, 2, 0, 0), 
(0, 2, -6, 1, 1, 0), 
(0, 2, -5, 0, 0, 1), $}
\footnotesize{$(0, 2, -4, 0, 0, 0), 
(0, 3, -6, 0, 1, 0), 
(0, 3, -5, 0, 0, 1)$} \\
$8$ & 
\footnotesize{$(0, 0, -8, 4, 0, 0), 
(0, 0, -7, 2, 1, 0), 
(0, 0, -6, 1, 2, 0), 
(0, 0, -6, 2, 0, 1), 
(0, 0, -5, 0, 1, 1), 
(0, 0, -5, 1, 0, 0), 
(0, 0, -4, 0, 0, 2), $}
\footnotesize{$(0, 0, -4, 0, 1, 0), 
(0, 1, -8, 3, 0, 0), 
(0, 1, -7, 2, 1, 0), 
(0, 1, -6, 0, 2, 0), 
(0, 1, -6, 1, 0, 1), 
(0, 1, -5, 0, 1, 1), 
(0, 1, -5, 1, 0, 0), $}
\footnotesize{$(0, 2, -7, 1, 1, 0), 
(0, 2, -6, 0, 2, 0), 
(0, 2, -6, 1, 0, 1), 
(0, 2, -5, 0, 0, 0), 
(0, 3, -6, 0, 0, 1), 
(0, 3, -5, 0, 0, 0)$}\\
$9$ &
\footnotesize{$(0, 0, -9, 4, 0, 0), 
(0, 0, -8, 3, 1, 0), 
(0, 0, -7, 1, 2, 0), 
(0, 0, -7, 2, 0, 1), 
(0, 0, -6, 0, 3, 0), 
(0, 0, -6, 1, 1, 1), 
(0, 0, -6, 2, 0, 0), $}
\footnotesize{$(0, 0, -5, 0, 0, 2), 
(0, 0, -5, 0, 1, 0), 
(0, 0, -4, 0, 0, 1), 
(0, 1, -8, 2, 1, 0), 
(0, 1, -7, 1, 2, 0), 
(0, 1, -7, 2, 0, 1), 
(0, 1, -6, 0, 1, 1), $}
\footnotesize{$(0, 1, -6, 1, 0, 0), 
(0, 1, -5, 0, 0, 2), 
(0, 1, -5, 0, 1, 0), 
(0, 2, -7, 0, 2, 0), 
(0, 2, -7, 1, 0, 1), 
(0, 2, -6, 0, 1, 1), 
(0, 2, -6, 1, 0, 0), $}
\footnotesize{$(0, 3, -6, 0, 0, 0)$} \\
$10$&
\footnotesize{$(0, 0, -9, 3, 1, 0), 
(0, 0, -8, 2, 2, 0), 
(0, 0, -8, 3, 0, 1), 
(0, 0, -7, 0, 3, 0), 
(0, 0, -7, 1, 1, 1), 
(0, 0, -7, 2, 0, 0), 
(0, 0, -6, 0, 2, 1), $}
\footnotesize{$(0, 0, -6, 1, 0, 2), 
(0, 0, -6, 1, 1, 0), 
(0, 0, -5, 0, 0, 1), 
(0, 0, -4, 0, 0, 0), 
(0, 1, -8, 1, 2, 0), 
(0, 1, -8, 2, 0, 1), 
(0, 1, -7, 0, 3, 0), $}
\footnotesize{$(0, 1, -7, 1, 1, 1), 
(0, 1, -7, 2, 0, 0), 
(0, 1, -6, 0, 0, 2), 
(0, 1, -6, 0, 1, 0), 
(0, 1, -5, 0, 0, 1), 
(0, 2, -7, 0, 1, 1), 
(0, 2, -7, 1, 0, 0), $}
\footnotesize{$(0, 2, -6, 0, 0, 2), 
(0, 2, -6, 0, 1, 0)$} \\
\bottomrule
\end{tabularx}
\caption{Weights appearing in $\Exterior^p \mathcal F^{\ast}$ for $\mathcal F = \mathcal E_{\varpi_6}^{\oplus 4} \oplus \mathcal O(1)$ over $E_6/P_3$ for $1 \le p \le 10$.}
\label{table_E6P3_all_weights}
\end{table}
Using \texttt{SageMath}, we compute the weights appearing in $\Exterior^p \mathcal F^{\ast}$ for all $1\le p \le 20$ and check that all of them are singular. 
Accordingly, by the Borel--Weil--Bott theorem, each irreducible equivariant bundle in $\Exterior^p \mathcal F^{\ast}$ is acyclic for $1 \le p \le 20$. 
Finally, we obtain
\[
H^{25}(E_6/P_3, \Exterior^{21} \mathcal F^{\ast})
= H^{25}(E_6/P_3, \mathcal O(-9)) \cong H^{0}(E_6/P_3, \mathcal O)^{\ast} \cong \C,
\]
which implies $H^4(Z,\mathcal O_Z) \cong \C$. This proves the claim for this case. 

For the remaining cases, again using the computer program \texttt{SageMath}~\cite{SAGE}, we compute the exterior powers of bundles and check whether they are acyclic or not. 
\end{proof}

\begin{proof}[Proof of Theorem~\ref{hyperkahler fourfold}]
From Proposition~\ref{prop_h0q_Z} and Corollary~\ref{criterion}, 
each fourfold with trivial canonical bundle classified in Theorem~\ref{classification} for exceptional homogeneous varieties of Picard number one is not hyperk\"{a}hler. 
\end{proof}

\begin{proof}[Proof of Corollary~\ref{corollary of main result}]
Theorem 1.1 of~\cite{Ben} says that if $Z$ is a hyperk\"{a}hler fourfold 
which is the zero locus of a general global section of completely reducible, globally generated, equivariant vector bundles over a Grassmannian, orthogonal or symplectic Grassmannian, 
then $Z$ is either of Beauville--Donagi type or of Debarre--Voisin type. 
Hence, by Theorem~\ref{hyperkahler fourfold} we get the conclusion.  
\end{proof}

\subsection{Computations of $h^{1, q}(Z)$}

First, recall the geometric meaning of the Hodge number $h^{1, 3}(Z)=h^{3, 1}(Z)$ of a Calabi--Yau fourfold $Z$. 
Since $Z$ has trivial canonical bundle, $H^1(Z, \Omega^3_Z) = H^1(Z, K_Z \otimes T_Z) = H^1(Z, T_Z)$. 
By Kodaira--Spencer deformation theory (cf.~\cite{Kodaira}), 
$h^{3, 1}(Z)$ is equal to the dimension of deformation parameter space of $Z$ 
because $Z$ is unobstructed from $H^2(Z, T_Z)=0$.
 
If $G/P$ is not a Hermitian symmetric space of compact type, then the tangent bundle $T_{G/P}$ of $G/P$ is reducible. 
Nevertheless, using a filtration by $P$-submodules, 
the Borel--Weil--Bott theorem can be applied to compute cohomology groups of reducible equivariant vector bundles. 
Any $P$-module $V$ can be decomposed into a direct sum of irreducible $L$-modules 
$V = W_0 \oplus W_1 \oplus \cdots \oplus W_t$, 
where $L$ is the reductive part of a Levi decomposition of $P=L U$. 
Since the unipotent radical $U$ is a normal subgroup of $P$ and $U$ acts on $V$ in triangular fashion, 
these irreducible $L$-modules can be arranged such that $U. W_i \subset W_{j}$ with $j > i$. 
Then $V$ has a filtration $0 \subset V_t \subset \cdots \subset V_1 \subset V_0 = V$ such that $V_i/V_{i+1} \cong W_i$. 

Let $E_i = G \times_P V_i$ and $F_i = G \times_P W_i$ be the equivariant vector bundles associated with $V_i$ and $W_i$ over $G/P$, respectively.
Then we get a short exact sequence of vector bundles 
$0 \to E_{i+1} \to E_i \to F_i \to 0$. 

Clearly, when $W_i = V_P(\lambda_i)^*$ for a $P$-dominant weight $\lambda_i$, we have $F_i = \mathcal E_{\lambda_i}$. 

\begin{definition} 
Let $E = G \times_P V$ be a (reducible) equivariant vector bundle over $G/P$. 
For an irreducible decomposition $V = W_0 \oplus W_1 \oplus \cdots \oplus W_t$ as $L$-modules, 
we denote a $P$-dominant weight $\lambda_i$ such that $W_i = V_P(\lambda_i)^*$. 
We define the set 
\[
\text{RegInd}(E) \coloneqq \{ \text{index}(\lambda_i + \rho) : \lambda_i + \rho \text{ is regular for } 0 \leq i \leq t \},
\] 
where 
$\text{index}(\lambda)$ is the minimum among the lengths~$\ell(w)$ of the Weyl group elements $w$ making $w(\lambda)$ regular dominant. 
\end{definition}

\begin{proposition}[{\cite[Section~3]{Griffiths63}}]
\label{Griffiths}
Let $E$ be a \textup{(}reducible\textup{)} equivariant vector bundle over $G/P$. 
If $q \notin \text{\rm RegInd}(E)$, then $H^q(G/P, E)=0$. 
\end{proposition}

\begin{proof}
The short exact sequence of vector bundles 
$0 \to E_{i+1} \to E_i \to F_i \to 0$ 
leads to the following long exact sequences 
\begin{eqnarray*}
& \cdots \to H^q(G/P, E_t) \to H^q(G/P, E_{t-1}) \to H^q(G/P, F_{t-1}) \to \cdots \\
& \cdots \to H^q(G/P, E_{t-1}) \to H^q(G/P, E_{t-2}) \to H^q(G/P, F_{t-2}) \to \cdots \\
& \cdots \\
& \cdots \to H^q(G/P, E_1) \to H^q(G/P, E) \to H^q(G/P, F_0) \to \cdots 
\end{eqnarray*}
By Theorem \ref{BBW}, if $q \notin  \text{\rm RegInd}(E)$, then we have 
\[
H^q(G/P, E_t) = H^q(G/P, F_t) = 0, H^q(G/P, F_{t-1}) =0, \dots , H^q(G/P, F_0) = 0.
\]  
Hence we get $H^q(G/P, E) = 0$ from the above exact sequences. 
\end{proof}

From now on, we only consider a rational homogeneous variety $G/P_k$ of Picard number one, that is, $P_k$ is the maximal parabolic subgroup associated to a simple root $\alpha_k$ of $G$.
Let ${\mathfrak{p}}_k$ be the maximal parabolic subalgebra of $\mathfrak g$ associated to the simple root $\alpha_k$.
Given an integer $\ell$,
$-m \leq \ell \leq m$, $\Phi_{\ell}$ denotes the set of all roots 
$\alpha=\sum_{q=1}^{r} c_q \alpha_q$
with the $k$th coefficient $c_k = \ell$. 
Define $$\mathfrak g_0 = \mathfrak h \oplus \bigoplus_{\alpha\in\Phi_0}
\mathfrak g_{\alpha},\qquad \mathfrak g_{\ell} =
\bigoplus_{\alpha\in\Phi_{\ell}} \mathfrak g_{\alpha}, \,\, \ell \neq 0,$$
where $\mathfrak h$ is a Cartan subalgebra of $\mathfrak g$. 
Then we have a graded Lie algebra 
$\mathfrak g = \mathfrak g_{-m} \oplus \cdots \oplus \mathfrak g_{-1} \oplus \mathfrak g_0 \oplus \mathfrak g_1 \cdots \oplus \mathfrak g_m$  
of depth $m$ associated with the simple root $\alpha_k$. 
Using the natural identification 
\[
T_o(G/P_k)= \mathfrak g / {\mathfrak{p}}_k \cong \mathfrak g_{-m} \oplus \cdots \oplus \mathfrak g_{-1},
\] 
we describe the tangent bundle $T_{G/P_k}$ and the cotangent bundle $\Omega_{G/P_k}$ via irreducible equivariant vector bundles.
Here, $o = eP \in G/P_k$ where $e$ is the identity element of $G$.

\begin{example}[No. 9 in Table~\ref{table1}]
Let $Z$ be the zero locus of a general global section of the line bundle $\mathcal O(3)$ over $G_2/P_2$. 
Since $\Phi_{1} = \{ \alpha_2, \alpha_1 + \alpha_2, 2\alpha_1 + \alpha_2, 3\alpha_1 + \alpha_2 \}$ and 
$\Phi_{2} = \{ 3\alpha_1 + 2\alpha_2 \}$ for the root system of $G_2$, 
\begin{equation}\label{eq_ToG2P2}
T_o(G_2/P_2) \cong \mathfrak g_{-2} \oplus \mathfrak g_{-1} = V_{P_2}(-3\alpha_1 - 2\alpha_2) \oplus V_{P_2}(- \alpha_2)
\end{equation}
has a filtration 
\[
0 \subset V_{P_2}(-\varpi_2) \subset V_{P_2}(-\varpi_2) \oplus V_{P_2}(3\varpi_1 -2 \varpi_2).
\]
Moreover, we get a short exact sequence 
\begin{equation}\label{eq_G2P2_Omega}
0 \to \mathcal E_{-\varpi_2} \to \Omega_{G_2/P_2} \to \mathcal E_{3\varpi_1 -2 \varpi_2} \to 0
\end{equation}
from $\Omega_{G_2/P_2} = G_2 \times_{P_2} T_o(G_2/P_2)^*$. 
By Proposition \ref{Griffiths}, $H^q(G_2/P_2, \Omega_{G_2/P_2})=0$ for $q \neq 1$ 
because $\text{\rm RegInd}(\Omega_{G_2/P_2})=\{ 1 \}$.

Moreover, tensoring the Koszul complex 
\begin{equation}\label{eq_Koszul_G2P2}
0 \to \mathcal O(-3) \to \mathcal O_{G_2/P_2} \to \mathcal O_Z \to 0
\end{equation} 
with $\Omega_{G_2/P_2}$, 
we obtain the exact sequence
\[
0 \to \Omega_{G_2/P_2}(-3) \to \Omega_{G_2/P_2} \to \Omega_{G_2/P_2}|_Z \to 0.
\] 
Using $H^q(G_2/P_2, \Omega_{G_2/P_2}(-3))=0$ for $q \neq 5$ from $\text{\rm RegInd}(\Omega_{G_2/P_2}(-3))=\{ 5 \}$, 
we see 
\begin{align}
H^1(Z, \Omega_{G_2/P_2}|_Z) &\cong H^1(G_2/P_2, \Omega_{G_2/P_2}) \cong \mathbb C, \label{eq_G2P2_H1_Omega}\\ 
H^4(Z, \Omega_{G_2/P_2}|_Z) & \cong H^5(G_2/P_2, \Omega_{G_2/P_2}(-3)) \cong V_{G_2}(\varpi_2)^* \cong \mathbb C^{14},
\end{align}
and the other cohomologies vanish. 
Thus the conormal sequence $0 \to \mathcal O(-3)|_Z  \to \Omega_{G_2/P_2}|_Z \to \Omega_Z \to 0$ 
leads to the following exact sequences
\begin{align}
& 0 \to H^0(Z,\Omega_Z) \to H^1(Z, \mathcal{O}(-3)|_Z) \to H^1(Z, \Omega_{G_2/P_2}|_Z) \to H^1(Z, \Omega_Z) \to H^2(Z, \mathcal{O}(-3)|_Z) \to 0;  \label{eq_G2P2_long_exact_seq_from_conormal1}
 \\
& 0 \to H^2(Z, \Omega_Z) \to H^3(Z, \mathcal{O}(-3)|_Z) \to 0; \label{eq_G2P2_long_exact_seq_from_conormal2} \\
&0 \to H^3(Z, \Omega_Z) \to H^4(Z, \mathcal O(-3)|_Z) \to H^4(Z, \Omega_{G_2/P_2}|_Z) \to H^4(Z, \Omega_Z) \to H^5(Z, \mathcal O(-3)|_Z) \to 0. \label{eq_G2P2_long_exact_seq_from_conormal}
\end{align} 
Because $H^4(Z, \Omega_Z) = H^{1, 4}(Z) \cong H^{4, 1}(Z) = H^1(Z, \mathcal O_Z) =0$, 
we conclude that $H^5(Z, \mathcal O(-3)|_Z)=0$ and the sequence~\eqref{eq_G2P2_long_exact_seq_from_conormal} becomes
\[
\begin{tikzcd}[column sep = 0.5cm, row sep = 0.3cm]
0 \rar & H^3(Z, \Omega_Z) \rar 
& H^4(Z, \mathcal O(-3)|_Z) \rar 
& H^4(Z, \Omega_{G_2/P_2}|_Z) \rar 
& H^4(Z, \Omega_Z) \rar 
& H^5(Z, \mathcal O(-3)|_Z) \rar 
& 0. \\
&
&
& \mathbb C^{14} \arrow[u, labels = description, "\rotatebox{90}{$\cong$}"]
& 0 \arrow[u, equal]
& 0 \arrow[u, equal]
\end{tikzcd}
\]
Accordingly, we have 
\begin{equation}\label{eq_G2P2}
H^3(Z, \Omega_Z) \cong H^4(Z, \mathcal O(-3)|_Z) / H^4(Z, \Omega_{G_2/P_2}|_Z). 
\end{equation}

Similarly, tensoring the Koszul complex~\eqref{eq_Koszul_G2P2} with $\mathcal O(-3)$ 
we obtain 
\[
0 \to \mathcal O(-6) \to \mathcal O(-3) \to \mathcal O(-3)|_Z  \to 0.
\] 
The straightforward applications of the Borel--Weil--Bott theorem say that 
\[
\begin{split}
H^5(G_2/P_2, \mathcal O(-3)) &=V_{G_2}(0)^* \cong \mathbb C, \\
H^5(G_2/P_2, \mathcal O(-6)) &=V_{G_2}(3\varpi_2)^* \cong \mathbb C^{273}.
\end{split}
\] 
Since $\text{RegInd}(\mathcal O(-3)) = \text{RegInd}(\mathcal O(-6)) =  \{5\}$, we obtain
\[
\begin{tikzcd}[column sep = 0.5cm, row sep = 0.3cm]
0 \rar & H^4(Z, \mathcal O(-3)|_Z) \rar 
& H^5(G_2/P_2, \mathcal O(-6)) \rar 
& H^5(G_2/P_2, \mathcal O(-3)) \rar 
& H^5(Z, \mathcal O(-3)|_Z) \rar 
& 0. \\
&
& \C^{273} \arrow[u, labels = description, "\rotatebox{90}{$\cong$}"]
& \C \arrow[u, labels = description, "\rotatebox{90}{$\cong$}"]
& 0 \arrow[u, equal]
\end{tikzcd}
\]
Hence $H^4(Z, \mathcal O(-3)|_Z) \cong H^5(G_2/P_2, \mathcal O(-6)) / H^5(G_2/P_2, \mathcal O(-3)) \cong \mathbb C^{272}$, and moreover, $H^q(Z, \mathcal O(-3)|_Z) = 0$ for $q \neq 4$.
Therefore, by~\eqref{eq_G2P2_H1_Omega} the exact sequences~\eqref{eq_G2P2_long_exact_seq_from_conormal1} and~\eqref{eq_G2P2_long_exact_seq_from_conormal2} provide
\[
H^1(Z,\Omega_Z) \cong H^1(Z, \Omega_{G_2/P_2}|_Z) \cong \C; \quad
H^2(Z,\Omega_Z) \cong H^3(Z,\mathcal{O}(-3)|_Z) = 0.
\]
Accordingly, we have $h^{1, 1}(Z)  = 1$ and $h^{1, 2}(Z) = 0$.
By considering~\eqref{eq_G2P2}, we get 
\[
h^{1, 3}(Z) = \dim H^3(Z, \Omega_Z) = \dim H^4(Z, \mathcal O(-3)|_Z) - \dim H^4(Z, \Omega_{G_2/P_2}|_Z) = 272 - 14 = 258.
\]  
\end{example}

\begin{example}[No. 8 in Table~\ref{table1}]
Let $Z$ be the zero locus of a general global section of the line bundle~$\mathcal F = \mathcal O(5)$ over $G_2/P_1$. 
The Lie algebra of $G_2$ has a gradation of depth $3$ associated with the simple root $\alpha_1$. 
Since $\Phi_{1} = \{ \alpha_1, \alpha_1 + \alpha_2 \}$, 
$\Phi_{2} = \{ 2\alpha_1 + \alpha_2 \}$, 
$\Phi_{3} = \{ 3\alpha_1 + \alpha_2, 3\alpha_1 + 2\alpha_2 \}$, 
we have the decomposition 
\[
\begin{split}
T_o(G_2/P_1) \cong \mathfrak g_{-3} \oplus \mathfrak g_{-2} \oplus \mathfrak g_{-1} 
&= V_{P_1}(-3\alpha_1 - \alpha_2) \oplus V_{P_1}(-2\alpha_1 - \alpha_2) \oplus V_{P_1}(- \alpha_1) \\
&= V_{P_1}(-3\varpi_1 + \varpi_2) \oplus V_{P_1}(-\varpi_1) \oplus V_{P_1}(-2\varpi_1 + \varpi_2)
\end{split}
\]
and this gives a filtration 
\[
0 \subset \mathfrak g_{-3} \subset \mathfrak g_{-3} \oplus \mathfrak g_{-2} \subset \mathfrak g_{-3} \oplus \mathfrak g_{-2} \oplus \mathfrak g_{-1} = T_o(G_2/P_1).
\] 
Putting $E\coloneqq G_2 \times_{P_1} (\mathfrak g_{-3} \oplus \mathfrak g_{-2})^*$,  
we get two short exact sequences
\begin{equation}\label{eq_G2P1_two_exact_sequences}
\begin{tikzcd}[column sep = 0.5cm, row sep = 0cm]
0 \rar & \mathcal E_{-3\varpi_1 + \varpi_2} \rar & E \rar & \mathcal E_{-\varpi_1} \rar & 0,  \\
0 \rar & E \rar & \Omega_{G_2/P_1} \rar & \mathcal E_{-2\varpi_1 + \varpi_2} \rar & 0
\end{tikzcd}
\end{equation}
from $\Omega_{G_2/P_1} = G_2 \times_{P_1} T_o(G_2/P_1)^*$. 
We notice that the weights $-\varpi_1$ and $-3\varpi_1+ \varpi_2$ are singular while $-2\varpi_1+ \varpi_2$ is regular with $\text{index}(-2\varpi_1+ \varpi_2) = 1$.
By Proposition~\ref{Griffiths}, $H^q(G_2/P_1, \Omega_{G_2/P_1})=0$ for $q \neq 1$ 
because $\text{\rm RegInd}(\Omega_{G_2/P_2})=\{ 1 \}$.

Tensoring the Koszul complex 
\begin{equation}\label{eq_Koszul_G2P1}
0 \to \mathcal O(-5) \to \mathcal O_{G_2/P_1} \to \mathcal O_Z \to 0
\end{equation} 
with $\Omega_{G_2/P_1}$, 
we obtain the exact sequence 
\[
0 \to \Omega_{G_2/P_1}(-5) \to \Omega_{G_2/P_1} \to \Omega_{G_2/P_1}|_Z \to 0.
\] 
Using $H^q(G_2/P_1, \Omega_{G_2/P_1}(-5))=0$ for $q \neq 5$ from $\text{\rm RegInd}(\Omega_{G_2/P_1}(-5))=\{ 5 \}$, 
we see 
\begin{align}
H^1(Z, \Omega_{G_2/P_1}|_Z) &= H^1(G_2/P_1, \Omega_{G_2/P_1}) \cong \mathbb C, \label{eq_G2P1_H4Z_0} \\ 
H^4(Z, \Omega_{G_2/P_1}|_Z) &= H^5(G_2/P_1, \Omega_{G_2/P_1}(-5))= H^5(G_2/P_1, E(-5)) \cong \mathbb C^{21} \label{eq_G2P1_H4Z}
\end{align} 
and the other cohomologies vanish:
\begin{equation}\label{eq_G2P1_Z_i}
H^q(Z, \Omega_{G_2/P_1}|_Z) = 0 \quad \text{ for } q = 0,2,3.
\end{equation} 
Here, we explain how to get $H^5(G_2/P_1, E(-5)) \cong \mathbb C^{21} $ more precisely. 
Tensoring the first exact sequence in~\eqref{eq_G2P1_two_exact_sequences} with $\mathcal F^{\ast} = \mathcal{O}(-5)$, 
we obtain the exact sequence $0 \to \mathcal E_{\varpi_2}(-8) \to E(-5) \to \mathcal O(-6) \to 0$. 
This induces 
\[
\begin{tikzcd}[column sep = 0.5cm, row sep = 0.3cm]
0 \rar & H^5(G_2/P_1, \mathcal E_{\varpi_2}(-8)) \rar & H^5(G_2/P_1, E(-5))  \rar & H^5(G_2/P_1, \mathcal O(-6)) \rar &0,\\
& \C^{7} \arrow[u, labels = description, "\rotatebox{90}{$\cong$}"]
&
& \C^{14} \arrow[u, labels = description, "\rotatebox{90}{$\cong$}"]
\end{tikzcd}
\]
and we know $H^5(G_2/P_1, \mathcal E_{\varpi_2}(-8)) = V_{G_2}(\varpi_1)^* \cong \mathbb C^{7}$, $H^5(G_2/P_1, \mathcal O(-6)) = V_{G_2}(\varpi_2)^* \cong \mathbb C^{14}$. 
This provides $H^5(G_2/P_1, E(-5))  \cong \C^{21}$.

Now the conormal sequence $0 \to \mathcal O(-5)|_Z  \to \Omega_{G_2/P_1}|_Z \to \Omega_Z \to 0$ and~\eqref{eq_G2P1_Z_i} lead to the following long exact sequences
\begin{align}
& 0 \to H^0(Z,\Omega_Z) \to H^1(Z, \mathcal{O}(-5)|_Z) \to H^1(Z, \Omega_{G_2/P_1}|_Z) \to H^1(Z, \Omega_Z) \to H^2(Z, \mathcal{O}(-5)|_Z) \to 0;  \label{eq_G2P1_long_exact_seq_from_conormal1}
 \\
& 0 \to H^2(Z, \Omega_Z) \to H^3(Z, \mathcal{O}(-5)|_Z) \to 0; \label{eq_G2P1_long_exact_seq_from_conormal2} \\
&0 \to H^3(Z, \Omega_Z) \to H^4(Z, \mathcal O(-5)|_Z) \to H^4(Z, \Omega_{G_2/P_1}|_Z) \to H^4(Z, \Omega_Z) \to H^5(Z, \mathcal O(-5)|_Z) \to 0. \label{eq_G2P1_long_exact_seq_from_conormal}
\end{align} 
From $H^4(Z, \Omega_Z)=0$, we get $ H^5(Z, \mathcal O(-5)|_Z) = 0$ and $H^3(Z, \Omega_Z) \cong H^4(Z, \mathcal O(-5)|_Z) / H^4(Z, \Omega_{G_2/P_1}|_Z)$. 
Tensoring the Koszul complex~\eqref{eq_Koszul_G2P1} with $\mathcal O(-5)$, 
we have the exact sequence 
\[
0 \to \mathcal O(-10) \to \mathcal O(-5) \to \mathcal O(-5)|_Z \to 0.
\] 
Since $\text{RegInd}( \mathcal O(-10)) = 
\text{RegInd}( \mathcal O(-5)) = \{5\}$, 
the computations $H^5(G_2/P_1, \mathcal O(-10)) = V_{G_2}(5\varpi_1)^* \cong \mathbb C^{378}$ and $H^5(G_2/P_1, \mathcal O(-5)) \cong \mathbb C$ 
imply 
\begin{align}
H^4(Z, \mathcal O(-5)|_Z) &\cong H^5(G_2/P_1, \mathcal O(-10)) / H^5(G_2/P_1, \mathcal O(-5)) \cong \mathbb C^{377}; \label{eq_G2P1_H4ZO-5} \\
H^q(Z, \mathcal O(-5)|_Z) &= 0 \quad \text{ for }q \neq 4. \label{eq_G2P1_H4ZO-5_2}
\end{align}
By~\eqref{eq_G2P1_H4Z_0} and~\eqref{eq_G2P1_H4ZO-5_2}, the exact sequences~\eqref{eq_G2P1_long_exact_seq_from_conormal1} and~\eqref{eq_G2P1_long_exact_seq_from_conormal2} provide
\[
H^1(Z,\Omega_Z) \cong H^1(Z, \Omega_{G_2/P_1}|_Z) \cong \C; \quad
H^2(Z,\Omega_Z) \cong H^3(Z,\mathcal{O}(-5)|_Z) = 0.
\]
Moreover,
the exact sequence~\eqref{eq_G2P1_long_exact_seq_from_conormal} becomes
\[
\begin{tikzcd}[column sep = 0.5cm, row sep = 0.3cm]
0 \rar & H^3(Z, \Omega_Z) \rar & H^4(Z, \mathcal O(-5)|_Z) \rar 
& H^4(Z, \Omega_{G_2/P_1}|_Z) \rar & H^4(Z, \Omega_Z) \rar 
& H^5(Z, \mathcal O(-5)|_Z) \rar & 0 \\
&
& \C^{377} \arrow[u, labels = description, "\rotatebox{90}{$\cong$}"]
& \C^{21} \arrow[u, labels = description, "\rotatebox{90}{$\cong$}"]
& 0 \arrow[u, equal]
& 0 \arrow[u, equal]
\end{tikzcd}
\]
Here, the isomorphisms come from \eqref{eq_G2P1_H4Z} and~\eqref{eq_G2P1_H4ZO-5}.
Accordingly, we get
\[
\begin{split}
h^{1,1}(Z) &= \dim H^1(Z, \Omega_Z) = 1, \\
h^{1,2}(Z) &= \dim H^2(Z, \Omega_Z) = 0, \\
h^{1, 3}(Z) &= \dim H^3(Z, \Omega_Z) = \dim H^4(Z, \mathcal O(-5)|_Z) - \dim H^4(Z, \Omega_{G_2/P_1}|_Z) = 377 - 21 = 356.
\end{split}
\]   

\end{example}

In the same way as above, we obtain the Hodge number $h^{1,3}(Z)$ of the other Calabi--Yau fourfolds as in Table~\ref{table1}. 

\begin{proposition}\label{prop_h1q_Z}
Let $G$ be a simple Lie group of exceptional type.
Let $G/P$ be a rational homogeneous variety of Picard number one, and $\mathcal F$ be a completely reducible, globally generated, equivariant vector bundle over $G/P$. 
If the zero locus $Z$ of a general global section of $\mathcal F$ is a fourfold with trivial canonical bundle, 
then $h^{1, 0}(Z) = h^{1, 2}(Z) = 0$, $h^{1, 1}(Z) = 1$, and $h^{1, 3}(Z)$ are given as in Table~\ref{table1}. 
\end{proposition}

\begin{remark}
By the Hodge decomposition theorem and Proposition~\ref{prop_h0q_Z}, 
the result $h^{1, 1}(Z) = 1$ yields $H^2(Z, \mathbb C) = H^{1, 1}(Z) \cong \mathbb C$; 
hence $Z$ has Picard number one and $H^2(Z, \mathbb C)$ is spanned by the restriction of a K\"{a}hler form on $G/P$.
\end{remark}

In the same way as Section~\ref{sec_classification_fourfolds}, we classify Calabi--Yau threefolds 
which are zero loci of general global sections of completely reducible equivariant vector bundles over exceptional homogeneous varieties of Picard number one. 
The Hodge numbers and the Euler--Poincar{\'e} characteristic $\chi$ of such Calabi--Yau $3$-folds are listed in Table~\ref{table2}.

\subsection{Computations of $h^{2, 2}(Z)$}
In this section, we introduce a way to compute the remaining Hodge number~$h^{2, 2}(Z)$.
In order to compute the Hodge number $h^{2, 2}(Z)$, 
we may employ the exact sequence 
\begin{equation}\label{eq_sequence_h22}
0 \to S^2 \mathcal F^*|_Z  \to (\mathcal F^* \otimes \Omega_{G/P})|_Z \to \Omega_{G/P}^2|_Z \to \Omega_Z^2 \to 0
\end{equation}
obtained from the second exterior power of the conormal sequence; 
see, for example,~\cite[Section~3.9.1]{FM21}. 
After determining the irreducible decomposition of each of $S^2 \mathcal F^* \otimes \Exterior^i \mathcal F^*$, $(\mathcal F^* \otimes \Omega_{G/P}) \otimes \Exterior^i \mathcal F^*$ 
and $\Omega_{G/P}^2 \otimes \Exterior^i \mathcal F^*$, 
several and rather long applications of the Borel--Weil--Bott theorem provide the Hodge number $h^{2, 2}(Z)$.  
We examine the above method in the following example. The interesting reader can use the same techniques for the remaining examples in Table~\ref{table1}. 

\begin{example}[No. 9 in Table~\ref{table1}]
Let $Z$ be the zero locus of a general global section of the line bundle $\mathcal F = \mathcal O(3)$ over $G_2/P_2$. In this case, the sequence~\eqref{eq_sequence_h22} becomes
\begin{equation}\label{eq_G2P2_h22}
0 \to \mathcal O(-6)|_Z \to \Omega_{G_2/P_2}(-3)|_Z \to \Omega_{G_2/P_2}^2|_Z \to \Omega_Z^2 \to 0.
\end{equation}
We compute $H^q(Z, \mathcal O(-6)|_Z)$, $H^q(Z, \Omega_{G_2/P_2}(-3)|_Z)$, and $H^q(Z, \Omega_{G_2/P_2}^2|_Z)$ by considering the Koszul complex~\eqref{eq_Koszul_G2P2}. 

\smallskip 
\noindent \textsf{\fbox{Step 1.}} We first consider $H^q(Z, \mathcal O(-6)|_Z)$. 
To compute the cohomology, tensoring the Koszul complex~\eqref{eq_Koszul_G2P2} with~$\mathcal O(-6)$ yields
\[
0 \to \mathcal O(-9) \to \mathcal O(-6) \to \mathcal O (-6)|_Z \to 0.
\]
Since $\text{RegInd}(\mathcal{O}(-6)) = \text{RegInd}(\mathcal O(-9)) = \{5\}$, using the Borel--Weil--Bott theorem, we obtain
\[
\begin{split}
&H^5(G_2/P_2, \mathcal O(-6)) \cong \C^{273}, \quad H^q(G_2/P_2,\mathcal O(-6)) = 0 \text{ for } q \neq 5, \\
&H^5(G_2/P_2,\mathcal O(-9)) \cong \C^{3542}, \quad H^q(G_2/P_2,\mathcal O(-9)) = 0 \text{ for } q \neq 5.
\end{split}
\]
Accordingly, we obtain
\begin{equation}\label{eq_G2P2_h22_1}
\begin{split}
H^4(Z, \mathcal O(-6)|_Z) &\cong H^5(G_2/P_2, \mathcal O(-9))/ H^5(G_2/P_2, \mathcal O(-6)) \cong \C^{3269}; \\
H^q(Z, \mathcal O(-6)|_Z) &= 0 \quad \text{ for } q \neq 4.
\end{split}
\end{equation}

\smallskip
\noindent \textsf{\fbox{Step 2.}} We consider  $H^q(Z, \Omega_{G_2/P_2}(-3)|_Z)$.
Tensoring the Koszul complex~\eqref{eq_Koszul_G2P2} with~$\Omega_{G_2/P_2}(-3)$, we obtain 
\begin{equation}\label{eq_G2P2_Otensor-3_Z}
0 \to \Omega_{G_2/P_2}(-6) \to \Omega_{G_2/P_2}(-3) \to \Omega_{G_2/P_2}(-3)|_Z \to 0.
\end{equation}
Tensoring the exact sequence~\eqref{eq_G2P2_Omega} with $\mathcal O(-3)$, we get
\[
0 \to \mathcal E_{-4 \varpi_2} \to \Omega_{G_2/P_2}(-3) \to \mathcal E_{3 \varpi_1 - 5 \varpi_2} \to 0. 
\]
Since the weight $3 \varpi_1 - 5 \varpi_2$ is singular and $\text{RegInd}(\mathcal E_{-4 \varpi_2}) = \{5\}$, applying the Borel--Weil--Bott theorem, all cohomologies $H^q(G_2/P_2, \mathcal E_{3 \varpi_1 - 5 \varpi_2})$ vanish and $H^5(G_2/P_2, \mathcal E_{-4 \varpi_2}) \cong \C^{14}$. Accordingly,
\begin{equation}\label{eq_G2P2_Otensor-3}
H^5(G_2/P_2, \Omega_{G_2/P_2}(-3)) \cong H^5(G_2/P_2, \mathcal E_{-4 \varpi_2}) \cong \C^{14}; \quad H^q(G_2/P_2, \Omega_{G_2/P_2}(-3)) = 0 \quad \text{ for } q \neq 5.
\end{equation}

Similarly, tensoring the exact sequence~\eqref{eq_G2P2_Omega} with $\mathcal O(-6)$, we get 
\[
0 \to \mathcal E_{-7 \varpi_2} \to \Omega_{G_2/P_2}(-6) \to \mathcal E_{3 \varpi_1 - 8 \varpi_2} \to 0. 
\]
By applying the Borel--Weil--Bott theorem again, we obtain
\begin{equation}\label{eq_G2P2_Otensor-6}
\begin{split}
&H^5(G_2/P_2,  \Omega_{G_2/P_2}(-6)) \cong  
H^5(G_2/P_2, \mathcal E_{-7 \varpi_2}) \oplus H^5(G_2/P_2,\mathcal E_{3 \varpi_1 - 8 \varpi_2}) 
\cong \C^{748} \oplus \C^{1547} \cong \C^{2295}; \\
&H^q(G_2/P_2,  \Omega_{G_2/P_2}(-6)) = 0 \quad \text{ for } q \neq 5. 
\end{split}
\end{equation}
Using the computations~\eqref{eq_G2P2_Otensor-3} and~\eqref{eq_G2P2_Otensor-6}, from the exact sequence~\eqref{eq_G2P2_Otensor-3_Z} we have
\[
\begin{tikzcd}[column sep = 0.5cm, row sep = 0.3cm]
0 \rar & H^4(Z, \Omega_{G_2/P_2}(-3)|_Z) \rar
& H^5(G_2/P_2, \Omega_{G_2/P_2}(-6)) \rar
& H^5(G_2/P_2, \Omega_{G_2/P_2}(-3)) \rar
& 0 \\
&
& \C^{2295} \arrow[u, labels = description, "\rotatebox{90}{$\cong$}"]
& \C^{14} \arrow[u, labels = description, "\rotatebox{90}{$\cong$}"]
\end{tikzcd}
\]
Accordingly, we obtain
\begin{equation}\label{eq_G2P2_h22_2}
H^4(Z, \Omega_{G_2/P_2}(-3)|_Z) \cong \C^{2281}; \quad H^q(Z, \Omega_{G_2/P_2}(-3)|_Z) = 0 \quad \text{ for } q \neq 4.
\end{equation}

\smallskip 
\noindent \textsf{\fbox{Step 3.}} We consider $H^q(Z, \Omega_{G_2/P_2}^2|_Z)$. 
Tensoring the Koszul complex~\eqref{eq_Koszul_G2P2} with $\Omega_{G_2/P_2}^2$, we obtain
\begin{equation}\label{eq_G2P2_O2tensor}
0 \to \Omega_{G_2/P_2}^2(-3) \to \Omega_{G_2/P_2}^2 \to \Omega_{G_2/P_2}^2|_Z \to 0.
\end{equation}
Using $\Omega^2_{G_2/P_2} = \Exterior^2 \Omega_{G_2/P_2} 
= G_2 \times_{P_2} \Exterior^2 T_o(G_2/P_2)^{*}$ and the description of $T_o(G_2/P_2)$ in~\eqref{eq_ToG2P2}, we get a short exact sequence
\begin{equation}\label{eq_G2P2_O2_G2P2}
0 \to \mathcal E_{3 \varpi_1 - 3 \varpi_2} \to \Omega_{G_2/P_2}^2 \to \mathcal E_{-\varpi_2} \oplus \mathcal E_{4\varpi_1 - 3 \varpi_2} \to 0.
\end{equation}
By applying the Borel--Weil--Bott theorem on vector bundles $\mathcal E_{3 \varpi_1 - 3 \varpi_2}$, $\mathcal E_{-\varpi_2}$, and $\mathcal E_{4\varpi_1 - 3 \varpi_2}$, we obtain 
\[
\begin{split}
&H^q(G_2/P_2, \mathcal E_{3 \varpi_1 - 3 \varpi_2} ) = 0 \quad \text{ for all }q, \\
&H^2(G_2/P_2, \mathcal E_{-\varpi_2} \oplus \mathcal E_{4\varpi_1 - 3 \varpi_2} ) 
\cong \C; \quad H^q(G_2/P_2, \mathcal E_{-\varpi_2} \oplus \mathcal E_{4\varpi_1 - 3 \varpi_2} ) = 0 \quad \text{ for }q \neq 2.
\end{split}
\]
Accordingly, we have
\begin{equation}\label{eq_G2P2_h22_3_1}
H^2(G_2/P_2,\Omega_{G_2/P_2}^2) \cong \C; \quad
H^q(G_2/P_2,\Omega_{G_2/P_2}^2) = 0 \quad \text{ for } q \neq 2.
\end{equation}

On the other hand, tensoring the exact sequence~\eqref{eq_G2P2_O2_G2P2} with $\mathcal O(-3)$, we obtain
\[
0 \to \mathcal E_{3 \varpi_1 - 6 \varpi_2} \to \Omega_{G_2/P_2}^2(-3) \to \mathcal E_{-4\varpi_2} \oplus \mathcal E_{4\varpi_1 - 6 \varpi_2} \to 0.
\]
Using the Borel--Weil--Bott theorem, we get  
\[
\begin{split}
& H^5(G_2/P_2, \mathcal E_{3 \varpi_1 - 6 \varpi_2}) \cong \C^{77}; \quad 
H^q(G_2/P_2, \mathcal E_{3 \varpi_1 - 6 \varpi_2}) = 0 \quad \text{ for } q \neq 5, \\
& H^5(G_2/P_2, \mathcal E_{-4\varpi_2} \oplus \mathcal E_{4\varpi_1 - 6 \varpi_2}) \cong \C^{14}; \quad
H^q(G_2/P_2, \mathcal E_{-4\varpi_2} \oplus \mathcal E_{4\varpi_1 - 6 \varpi_2}) = 0 \quad \text{ for } q \neq 5.
\end{split}
\]
Therefore, we have
\begin{equation}\label{eq_G2P2_h22_3_2}
H^5(G_2/P_2, \Omega^2_{G_2/P_2}(-3)) \cong \C^{91}; 
\quad H^q(G_2/P_2, \Omega^2_{G_2/P_2}(-3)) = 0 \quad \text{ for } q \neq 5.
\end{equation}
Using the computations~\eqref{eq_G2P2_h22_3_1} and~\eqref{eq_G2P2_h22_3_2}, from the exact sequence~\eqref{eq_G2P2_O2tensor} we obtain
\begin{equation}\label{eq_G2P2_h22_3}
H^q(Z, \Omega_{G_2/P_2}^2|_Z) \cong \begin{cases}
\C & \text{ if } q = 2, \\
\C^{91} & \text{ if } q = 4, \\
0 & \text{ otherwise}.
\end{cases}
\end{equation}

\smallskip 
From the sequence~\eqref{eq_G2P2_h22}, we obtain two short exact sequences 
\begin{align}
0 \to \mathcal O (-6)|_Z \to & \Omega_{G_2/P_2}(-3)|_Z \to K \to 0, \label{eq_sequence_a} \\
0 \to K \to & \Omega_{G_2/P_2}^2|_Z \to \Omega_Z^2 \to 0. \label{eq_sequence_b}
\end{align}
We note that $h^{1,2}(Z) = h^{0,2}(Z)= 0$ holds by Propositions~\ref{prop_h0q_Z} and~\ref{prop_h1q_Z}. Using the Hodge symmetry $h^{p,q}(Z) = h^{q,p}(Z)$ and the Serre duality $h^{p,q}(Z) = h^{4-p,4-q}(Z)$, we obtain
\[
h^{2,3}(Z) = h^{3,2}(Z) = h^{1,2}(Z) = 0, \quad
h^{2,4}(Z) = h^{4,2}(Z) = h^{0,2}(Z) = 0.
\]
Accordingly, $H^3(Z,\Omega_Z^2) = H^4(Z,\Omega_Z^2) = 0$ and the sequence~\eqref{eq_sequence_b} provides 
\begin{equation}\label{eq_G2P2_h22_H4ZK}
H^4(Z,K) \cong H^4(Z, \Omega^2_{G_2/P_2}|_Z) \cong \C^{91}.
\end{equation}
Here, the last isomorphism comes from~\eqref{eq_G2P2_h22_3}. 

On the other hand, because $H^q(Z, \mathcal O(-6)|Z) = H^q(Z, \Omega_{G_2/P_2}(-3)|Z) = 0$ for $q \neq 4$ by~\eqref{eq_G2P2_h22_1} and~\eqref{eq_G2P2_h22_2}, the sequence~\eqref{eq_sequence_a} provides $H^q(Z,K) = 0$ for $q \neq 3,4$ and 
\[
\begin{tikzcd}[column sep = 0.5cm, row sep = 0.3cm]
0 \rar & H^3(Z,K) \rar 
& H^4(Z, \mathcal O(-6)|Z) \rar 
& H^4(Z, \Omega_{G_2/P_2}(-3)|_Z) \rar 
& H^4(Z, K) \rar & 0. \\
&
& \C^{3269} \arrow[u, labels = description, "\rotatebox{90}{$\cong$}"]
& \C^{2281} \arrow[u, labels = description, "\rotatebox{90}{$\cong$}"]
& \C^{91} \arrow[u, labels = description, "\rotatebox{90}{$\cong$}"]
\end{tikzcd}
\]
Here, the isomorphisms come from~\eqref{eq_G2P2_h22_1},~\eqref{eq_G2P2_h22_2}, and~\eqref{eq_G2P2_h22_H4ZK}, respectively. 
Therefore, we obtain 
\[
H^3(Z,K) \cong \C^{1079}. 
\]
From the sequence~\eqref{eq_sequence_b}, we have
\[
\begin{tikzcd}[column sep = 0.5cm, row sep = 0.3cm]
0 \rar & H^2(Z,\Omega_{G_2/P_2}^2|_Z) \rar
& H^2(Z,\Omega_Z^2) \rar
& H^3(Z,K) \rar
& 0. \\
& \C \arrow[u, labels = description, "\rotatebox{90}{$\cong$}"]
& 
& \C^{1079} \arrow[u, labels = description, "\rotatebox{90}{$\cong$}"]
\end{tikzcd}
\]
Therefore, we conclude $h^{2,2}(Z) = 1079+1 = 1080$. 

We obtain the Hodge diamond as follows:
\[
\begin{tikzcd}[column sep = -0.3cm, row sep = 0, baseline=-0.5ex]
&&&& h^{0,0} \\
&&& h^{1,0} && h^{0,1} \\
&& h^{2,0} && h^{1,1} && h^{0,2} \\
& h^{3,0} && h^{2,1} && h^{1,2} && h^{0,3} \\
h^{4,0} && h^{3,1} && h^{2,2} && h^{1,3} && h^{0,4} \\
& h^{4,1} && h^{3,2} && h^{2,3} && h^{1,4} \\
&& h^{4,2} && h^{3,3} && h^{2,4} \\
&&& h^{4,3} && h^{3,4} \\
&&&& h^{4,4}
\end{tikzcd}
\quad = \quad
\begin{tikzcd}[column sep = 0, row sep = 0, baseline=-0.5ex]
&&&& 1 \\
&&& 0 && 0 \\
&& 0 && 1 && 0 \\
& 0 && 0 && 0 && 0 \\
1 && 258 && 1080 && 258 && 1 \\
& 0 && 0 && 0 && 0\\
&& 0 && 1 && 0 \\
&&& 0 && 0 \\
&&&& 1
\end{tikzcd}
\]
The Euler--Poincar\'e characteristic $\chi(Z)$ is 
\[
\chi(Z) = 1 + 1 + (1 + 258 + 1080 + 258 + 1) + 1 + 1 
= 1602.
\]
\end{example}

\begin{remark}
In~\cite{GHL14}, the Hodge numbers are computed for complete intersection Calabi--Yau fourfolds.
On the other hand, as is mentioned in Remark~\ref{rmk_Z_for_G2}, the fourfold $Z$ given in No. 8 is a complete intersection Calabi--Yau fourfold in $\mathbb{P}^6$. 
Following the computational result in~\cite{GHL14} (indeed, $Z$ corresponds to `\texttt{MATRIX NUMBER : 2}' in the file), we obtain $h^{2,2}(Z) = 1472$, and moreover, the Euler--Poincar{\'e} characteristic~$\chi(Z)$ is~$2190$. 
\end{remark}

\vskip 2em 


\providecommand{\bysame}{\leavevmode\hbox to3em{\hrulefill}\thinspace}
\providecommand{\MR}{\relax\ifhmode\unskip\space\fi MR }
\providecommand{\MRhref}[2]{%
  \href{http://www.ams.org/mathscinet-getitem?mr=#1}{#2}
}
\providecommand{\href}[2]{#2}

\end{document}